\documentclass[reqno, 11pt]{amsart}

\usepackage[margin=1in]{geometry}
\usepackage[utf8]{inputenc}

\usepackage[hypertexnames=false,hyperfootnotes=false,colorlinks=true,linkcolor=blue,%
citecolor=purple,filecolor=magenta,urlcolor=cyan,unicode,linktocpage=true]{hyperref}
\usepackage{float}
\usepackage{scrextend}
\usepackage{enumitem}
\usepackage{amsmath, amssymb, amsrefs}
\usepackage{latexsym}
\usepackage{IEEEtrantools}
\usepackage{amsthm}
\usepackage{stmaryrd}
\SetSymbolFont{stmry}{bold}{U}{stmry}{m}{n}
\usepackage{bbm}
\usepackage{mathtools}
\usepackage{tikz,xypic}
\usepackage{tikz-cd}
\usepackage{ytableau}
\usepackage{cleveref}
\usepackage{caption}
\usepackage{amsthm}
\usepackage{todonotes}

\BibSpec{article}{%
+{}{\sc \PrintAuthors} {author}
+{,}{ \textrm} {title}
+{,}{ \textit} {journal}
+{,}{ } {volume}
+{}{ \IfEmptyBibField{journal}{}{\PrintDate{date}}} {transition}
+{,}{ } {pages}
+{,}{ } {status}
+{.}{} {transition}
+{}{ Preprint available at \tt} {eprint}
+{.}{} {transition}
}

\BibSpec{book}{%
+{}{\sc \PrintAuthors} {author}
+{,}{ \textit} {title}
+{.}{ } {series}
+{,}{ } {volume}
+{.}{ } {publisher}
+{,}{ } {place}
+{,}{ } {date}
+{.}{} {transition}
}

\BibSpec{collection.article}{
+{} {\sc \PrintAuthors} {author}
+{,} { \textrm } {title}
+{,} { in \it \PrintConference} {conference}
+{,} { pp.~} {pages}
+{.} {\PrintBook} {book}
+{,} { \PrintDateB} {date}
+{.} {} {transition}
}

\BibSpec{innerbook}{
+{.} { \emph} {title}
+{.} { } {part}
+{:} { \emph} {subtitle}
+{.} { } {series}
+{,} { } {volume}
+{.} { Edited by \PrintNameList} {editor}
+{.} { Translated by \PrintNameList}{translator}
+{.} { \PrintContributions} {contribution}
+{.} { } {publisher}
+{.} { } {organization}
+{,} { } {address}
+{,} { \PrintEdition} {edition}
+{,} { \PrintDateB} {date}
+{.} { } {note}
}

\newcommand{\ri}{\mathrm{i}}
\newcommand{\re}{\mathrm{e}}

\usepackage{anyfontsize}
\usepackage[lining]{libertine}%
\usepackage{courier}
\usepackage[T1]{fontenc}
\usepackage[utf8]{inputenc}
\usepackage{csquotes}
\makeatletter
\renewcommand*\libertine@figurestyle{LF}
\makeatother
\usepackage[libertine,libaltvw,liby]{newtxmath}
\usepackage{microtype}
\setlist[enumerate,1]{label=(\roman*),itemsep=0.9ex}
\setlist[itemize]{itemsep=0.9ex}
\renewcommand{\define}[1]{\textit{#1}}

\newcommand{\mc}{\mathcal}
\newcommand{\mr}{\mathrm}
\newcommand{\mf}{\mathfrak}

\newcommand{\Z}{\mathbb{Z}}

\renewcommand{\phi}{\varphi}

\DeclareMathOperator{\hhh}{H}
\newcommand*\circleed[1]{\tikz[baseline=(char.base)]{
		\node[shape=circle,draw,inner sep=1pt] (char) {#1};}}

\newcommand{\qbinom}{\genfrac{[}{]}{0pt}{}}
\newcommand{\sstyle}{\scriptstyle}

\crefname{equation}{Eq.}{Eqs.}
\crefname{eqnarray}{Eq.}{Eqs.}
\crefname{algo}{Algorithm}{Algorithms}
\crefname{conj}{Conjecture}{Conjectures}
\crefname{lem}{Lemma}{Lemmas}
\crefname{thm}{Theorem}{Theorems}
\crefname{claim}{Claim}{Claims}
\crefname{rmk}{Remark}{Remarks}
\crefname{prop}{Proposition}{Propositions}
\crefname{section}{Section}{Sections}
\crefname{appendix}{Appendix}{Appendices}
\crefname{cor}{Corollary}{Corollaries}
\crefname{figure}{Figure}{Figures}
\crefname{table}{Table}{Tables}
\crefname{example}{Example}{Examples}
\crefname{prob}{Problem}{Problems}
\crefname{assm}{Assumption}{Assumptions}
\crefname{defn}{Definition}{Definitions}

\newcommand{\cO}{\mathcal{O}}

\newcommand{\cM}{\mathcal M}

\newcommand{\bbO}{\mathbb{O}}

\newcommand{\bbZ}{\mathbb{Z}}
\newcommand{\bbR}{\mathbb{R}}
\newcommand{\bbC}{\mathbb{C}}
\newcommand{\bbP}{\mathbb{P}}

\renewcommand{\l}{\left}
\renewcommand{\r}{\right}

\def\beq{\begin{equation}}                     %  
\def\eeq{\end{equation}}                       % 
\def\bea{\begin{eqnarray}}                     %         % 
\def\eea{\end{eqnarray}}
\def\bary{\begin{array}} 
\def\eary{\end{array}} 
\def\ben{\begin{enumerate}} 
\def\een{\end{enumerate}}
\def\bit{\begin{itemize}} 
\def\eit{\end{itemize}}
\def\nn{\nonumber}

\theoremstyle{plain}

\newtheorem{thm}{Theorem}[section]
\newtheorem{lem}[thm]{Lemma}

\newtheorem{prop}[thm]{Proposition}
\newtheorem{conj}[thm]{Conjecture}

\newtheorem*{conj*}{Conjecture}
\newtheorem{cor}[thm]{Corollary}
\newtheorem*{cor*}{Corollary}

\theoremstyle{definition}

\newtheorem{rmk}[thm]{Remark}

\title{On quasi-tame Looijenga pairs}
\author{Andrea Brini}
\address{A.B.: University of Sheffield, School of Mathematics and Statistics, S3 7RH, Sheffield, United Kingdom\\
on leave from CNRS, DR~13, Montpellier, France}
\email{a.brini@sheffield.ac.uk}

\author{Yannik Schuler}
\address{Y.S.: University of Sheffield, School of Mathematics and Statistics, S3 7RH, Sheffield, United Kingdom}
\email{yschuler1@sheffield.ac.uk}

\begin{document}

\begin{abstract}
    We prove a conjecture of Bousseau, van Garrel and the first-named author relating, under suitable positivity conditions, the higher genus maximal contact log~Gromov--Witten invariants of Looijenga pairs to other curve counting invariants of Gromov--Witten/Gopakumar--Vafa type. The proof consists of a closed-form $q$-hypergeometric resummation of the quantum tropical vertex calculation of the log~invariants in presence of infinite scattering. The resulting identity of $q$-series appears to be new and of independent combinatorial interest. 
    
\end{abstract}

	\maketitle
	
\section{Introduction}
	
	\subsection{Quasi-tame Looijenga pairs}
	
	A Looijenga pair $Y(D) \coloneqq (Y,D)$ is the datum of a smooth rational complex projective surface $Y$ and an anticanonical singular curve $D \in |-K_Y|$. A Looijenga pair $Y(D)$ is called {\it nef} if the singular curve $D$ is a
	simple normal crossings divisor $D=\cup_{i=1}^l D_i$ with each $D_i$ smooth, irreducible, and nef\footnote{Since we require $D$ to be singular, an $l$-component nef Looijenga pair must have $l>1$.} for all $i=1, \dots, l$. A {\it tame} Looijenga pair is a nef pair with either $l>2$, or $D_i^2>0$ for all $i$. Writing $E_{Y(D)} \coloneqq \mathrm{Tot}(\oplus_i(\cO_Y(-D_i))$, we will say that a nef pair $Y(D)$ is {\it quasi-tame} if there exists a tame pair $Y'(D')$ such that $E_{Y(D)}$ is deformation-equivalent to $E_{Y'(D')}$. By definition, there is an obvious sequence of nested  inclusions

$$\hbox{nef Looijenga pairs} \supset \hbox{quasi-tame Looijenga pairs} \supset \hbox{tame Looijenga pairs}~. $$
\medskip

Looijenga pairs have been the focus of much attention lately due to their intertwined role in 
mirror symmetry for surfaces \cites{GHKlog, MR3518552, MR3563248, yu2016enumeration,barrott2020explicit,MR4048291, hacking2020homological, Bardwell-Evans:2021tsz} and the study of cluster varieties \cites{GHKclu, MR3954363, zhou2019weyl}.
In a recent series of papers \cite{Bousseau:2019bwv, Bousseau:2020fus, Bousseau:2020ryp}, the log~Gromov--Witten theory of quasi-tame pairs was further conjectured to be at the centre of a web of correspondences relating it to several enumerative theories. We recall the relevant context and fix notation below. 

\subsection{Enumerative theories} 

The authors of \cite{Bousseau:2020fus} consider four different geometries, and associated enumerative invariants, attached to the datum of a quasi-tame Looijenga pair $Y(D)$:

\ben  
\item %log GW
 the log Calabi--Yau surface obtained by viewing $Y(D)$ as a log-scheme for the divisorial log structure induced by $D$. For a given genus $g$ and effective curve class $d\in\hhh_2(Y,\bbZ)$ with $d\cdot D_i>0$ for all $i=1,\ldots,l$, the corresponding set of invariants are the log Gromov--Witten invariants \cites{Chen14,AbramChen14, GS13} of $Y(D)$ with maximal tangency at each component $D_i$, $l-1$ point insertions on the surface, and one insertion of the top Chern class of the Hodge bundle:
 \beq
 N^{\rm log}_{g,d}(Y(D)) \coloneqq \int_{[\cM^{\rm log}_{g,l-1}(Y(D),d)]^{\rm vir}} \prod_{i=1}^{l-1} \mathrm{ev}_i^* [\mathrm{pt}_Y] (-1)^g \lambda_g,
 \label{eq:loggw}
 \eeq 
or equivalently, their all-genus generating function
\beq
\mathbb{N}^{\rm log}_d(Y(D))(\hbar) \coloneqq
{\left( 
2 \sin \left( \frac{\hbar}{2} \right) \right)^{2-l}}
%(-\ri)^{2-l} [1]_q^{2-l}
\sum_{g \geqslant 0} N^{\rm log}_{g,d} (Y(D)) \hbar^{2g-2+l}\,;
\label{eq:GWexp_intro}
\eeq
\item the quasi-projective Calabi--Yau variety $E_{Y(D)}$, and its genus zero local Gromov--Witten invariants \cite{Chiang:1999tz,MR2276766}
\beq 
N_d(E_{Y(D)}) \coloneqq  \int_{[\cM_{0,l-1}(E_{Y(D)},d)]^{\rm vir}} \prod_{i=1}^{l-1} \mathrm{ev}_i^* [\mathrm{pt}_Y] 
\eeq 
and local Gopakumar--Vafa invariants \cite{Klemm:2007in,MR3739228}
\beq 
\mathrm{GV}_{d}(E_{Y(D)}) \coloneqq \sum_{k | d} \frac{\mu(k)}{k^{4-l}} N_{d}(E_{Y(D)})
\eeq 
where $\mu$ is the M\"{o}bius function;
\item the quasi-projective Calabi--Yau threefold $ \mathrm{Tot}(\cO_{Y \setminus \cup_{i<l} D_i}(-D_l))$ equipped with a disjoint union of $l-1$ Lagrangians $L_i$ fibred over real curves in $D_i$, $i<l$, as defined in \cite[Construction~6.4]{Bousseau:2020fus}: 
    \[
    Y^{\rm op}(D) \coloneqq  \begin{pmatrix}{\rm Tot}\begin{pmatrix}\cO(-D_l)\to Y\setminus\left(D_1\cup\cdots\cup D_{l-1}\right)\end{pmatrix},L_1\sqcup\cdots\sqcup L_{l-1}\end{pmatrix},
    \]
The respective invariants, for a given relative homology degree $d\in \hhh_2(Y^{\rm op}(D), \bbZ)$, are the open Gromov--Witten counts
\bea
O_{g,d}(Y^{\rm op}(D)) &\coloneqq & \int_{[\cM_{g}(Y^{\rm op}(D),d)]^{\rm vir}} 1, \nn \\
\mathbb{O}_d(Y^{\rm op}(D))(\hbar) &\coloneqq & \sum_{g\geq 0} \hbar^{2g+l-3} O_{g,d}(Y^{\rm op}(D)) 
\label{eq:opengw}
\eea 
virtually enumerating genus-$g$ open Riemann surfaces with $l-1$ connected components of the boundary ending on the Lagrangians $L_i$, $i=1, \dots l-1$. Under relatively lax conditions\footnote{This was formalised as ``Property~O'' in \cite[Definition~6.3]{Bousseau:2020fus}: 
this is equivalent to requiring that $E_{Y(D)}$ is deformation-equivalent to  $E_{Y'(D')}$ with $Y'$ a toric weak Fano surface, $D'_i$ a prime toric divisor for $i<l$, and $D'_l$ nef. 
All quasi-tame pairs with $l=2$ satisfy Property~O, and all non-tame quasi-tame pairs have $l=2$,
so this is safely assumed to hold throughout this paper.}, $Y^{\rm op}(D)$ can be deformed to a singular Harvey--Lawson (Aganagic--Vafa) Lagrangian pair with $L_i \simeq \bbR^2 \times S^1$, for which open GW counts can be defined in the algebraic category \cite{Katz:2001vm, Li:2001sg} (see also \cite{Li:2004uf,Fang:2011qd}). Denoting by $w_i(d)$  the winding number of a relative degree-$d$ open stable map to $Y^{\rm op}(D)$ around the non-trivial homology circle in $L_i$, we will also consider the corresponding genus zero/all-genus  Labastida--Mari\~no--Ooguri--Vafa invariants \cite{Ooguri:1999bv, Labastida:2000yw, Labastida:2000zp}
\bea
\mathrm{LMOV}_{0,d}(Y^{\rm op}(D)) &=& \sum_{k | d} \frac{\mu(k)}{k^{4-l}} {O}_{0,d/k}(Y^{\rm op}(D)), \nn \\
\mathbb{LMOV}_{d}(Y^{\rm op}(D))(\hbar) &=& [1]_q^2 \prod_{i<l} \frac{w_i(d)}{[w_i(d)]_q} \sum_{k | d} \frac{\mu(k)}{k} \mathbb{O}_{d/k}(Y^{\rm op}(D))(k \hbar),
\eea 
where $[n]_q \coloneqq q^{n/2}-q^{-n/2}$, and $q=\re^{\ri\hbar}$;
\item %quiver DT for $l=2$, a non-commutative geometry given by 
for $l=2$, a symmetric quiver $\mathsf{Q}(Y(D))$ with adjacency matrix determined by $Y(D)$ \cite[Thm.~7.3]{Bousseau:2020fus}. For a given charge vector~$d$, the corresponding numbers are the numerical Donaldson--Thomas invariants $\mathrm{DT}_d^{\rm num}(\mathsf{Q}(Y(D)))$, defined as the formal Taylor coefficients of (the plethystic logarithm of) the generating series of Euler characteristics on the stack of representations of $\mathsf{Q}(Y(D))$. 
\een 

The constructions of \cite{Bousseau:2020fus} in particular identify the absolute homology of $Y(D)$, the relative homology of $Y^{\rm op}(D)$, and the free abelian group over the set of vertices of $\mathsf{Q}(Y(D))$,
$$ d \in \hhh_2(Y(D), \bbZ) \simeq \hhh_2(E_{Y(D)}, \bbZ) \simeq \hhh^{\rm rel}_2(Y^{\rm op}(D), \bbZ) \simeq \bbZ^{|(\mathsf{Q}(Y(D)))_0|}.$$ Under these identifications, the authors of \cite{Bousseau:2020fus} propose that the invariants above are essentially the same, as follows.

%identify homologies/dimension vectors. 

\begin{conj}[The genus zero log/local/open correspondence, \cite{Bousseau:2020fus,vGGR,Liu:2021eeb}]
The genus zero log, local, and open Gromov--Witten invariants associated to a quasi-tame Looijenga pair $Y(D)$ are related as 
\beq 
 N_{d}(E_{Y(D)})= O_{0,d}(Y^{\rm op}(D)) =
 \left( \prod_{j=1}^l \frac{(-1)^{d \cdot D_j -1}}{d \cdot D_j} \right) N_{0,d}^{\rm log}(Y(D)),
\eeq 
and, for the associated BPS invariants,
\beq 
 \mathrm{GV}_{d}(E_{Y(D)})= \mathrm{LMOV}_{0,d}(Y^{\rm op}(D)) \in \bbZ.
\eeq 
Moreover, if $l=2$, 
\beq 
\mathrm{DT}_d^{\rm num}(\mathsf{Q}(Y(D)))=|\mathrm{GV}_{d}(E_{Y(D)})|.
\eeq 
\label{conj:gzero}
\end{conj}

In \cite[Conj.~1.3]{Bousseau:2020fus}, the above is further extended to an identity between all-genus Gromov--Witten generating functions.

\begin{conj}[The higher genus log/open correspondence]
The higher genus log and open Gromov--Witten invariants associated to a quasi-tame Looijenga pair $Y(D)$ are related as
\beq 
 \bbO_d(Y^{\rm op}(D))(\hbar) = \left( \prod_{j=1}^{l-1} \frac{(-1)^{d \cdot D_j -1}}{d \cdot D_j} \right) \frac{(-1)^{d \cdot D_l -1}}{[d \cdot D_l]_q} \mathbb{N}_d^{\rm log}(Y(D))(\hbar).
 \label{eq:logopengen}
\eeq 
Moreover, 
\beq 
\mathbb{LMOV}_d(Y^{\rm op}(D))(\hbar) =
 [1]^2_q
 \left( \prod_{i=1}^{l} \frac{ 1 }{[ d \cdot D_i]_{q}} \right)
 \sum_{k | d}
\frac{(-1)^{ d/k \cdot D + l}\mu(k)}{[k]_{q}^{2-l} \,  k^{2-l} \,} \,
\mathbb{N}_{d/k}^{\rm log}(Y(D))(k \hbar)
\in \bbZ[q, q^{-1}]\,.
\label{eq:lmovlogBPS}
\eeq 
\label{conj:higherg}
\end{conj}

\medskip

\cref{conj:gzero} was proved in \cite[Thm.~1.4--1.6]{Bousseau:2020fus}. \cref{conj:higherg} was proved in \cite[Thm.~1.5~and~1.7]{Bousseau:2020fus} for tame $Y(D)$, and formulated as a conjecture for quasi-tame $Y(D)$ in \cite[Conj.~4.8]{Bousseau:2020fus}. Two non-tame, quasi-tame cases of this conjecture were subsequently proved in \cite{Kra21}. \\

In this paper, we establish a stronger statement from which \cref{conj:higherg} follows for any quasi-tame Looijenga pair $Y(D)$. We will provide a closed-form calculation and comparison of both sides of \eqref{eq:logopengen}, returning the two equalities in \cref{conj:higherg} as a corollary. 	Since the tame setting was already dealt with in \cite{Bousseau:2020fus}, we may restrict our attention here to non-tame, quasi-tame pairs alone, for which $l=2$. 
	We give here a slightly discursive version of the main result of this paper (see \cref{prop:NlogdP302scat,prop:Gclosedform,prop:birinv} for precise statements, and e.g. \eqref{eq:opendp302} and \eqref{eq:NlogdP302closedform} for explicit formulas).
	
	\begin{thm}
	Let $Y(D)$ be a non-tame, quasi-tame Looijenga pair $Y(D)$ and $d\in \hhh_2(Y,\bbZ)$. The log and open higher genus generating functions %\eqref{eq:GWexp_intro} and \eqref{eq:opengw} 
	$\mathbb{N}^{\rm log}_d(Y(D))(\hbar)$ and $\mathbb{O}_d(Y^{\rm op}(D))(\hbar)$ 
	are rational functions of $q=\re^{\ri \hbar}$, with zeroes and poles only at $q=0$, $\infty$, or at roots of unity. Furthermore, \cref{conj:higherg} holds.
 \label{thm:introhigherg}
	\end{thm}
	
		\subsection{Strategy of the proof}
		
Our task is  simplified by a number of circumstances, which 
%effectively boil down to considering some explicit calculations of higher genus log Gromov--Witten invariants for 
reduce \cref{thm:introhigherg} to the computation of
{\it one single example}. As explained in \cite{Fri16} and \cite[Prop.~2.2]{Bousseau:2020fus}, $Y(D)$ is fully determined by the self-intersections $(D_1^2, D_2^2)$, and when considering specific examples we will shorten notation by writing $Y(D^2_1,D_2^2)$ for $Y(D)$. 
Let $\pi:\mathrm{dP}_r \to \bbP^2$ be the blow-up of the plane at $r\geq 0$ points $\{P_1, \dots, P_r\}$, and write $H \coloneqq \pi^*c_1(O_{\bbP^2}(1))$, $E_i\coloneqq [\pi^{-1}(P_i)]$. Up to deformation and permutation of ${D_1, D_2}$ and ${E_1, \dots, E_r}$, there is a unique non-tame, quasi-tame pair of maximal Picard number given by $Y=\mathrm{dP}_3$ and $D_1=H-E_1$, $D_2=2H-E_2-E_3$, so that $(D_1^2, D_2^2)=(0,2)$. \\

By \cite[Prop~2.2]{Bousseau:2020fus}, any other non-tame, quasi-tame pair $Y(D)$ is the result of a toric contraction $\pi':\mathrm{dP}_3(0,2) \to Y(D)$. Therefore, as we recall in \cref{prop:birinv} below, the blow-up formulas for log and open invariants \cite[Prop.~4.3 and 6.9] {Bousseau:2020fus} imply that it suffices to check that \cref{conj:higherg} holds for the single case $Y(D)=\mathrm{dP}_3(0,2)$. 
%The blow-up formulas for log \cite[Prop.~4.3]{Bousseau:2020fus} and open \cite[Prop.~6.9]{Bousseau:2020fus} Gromov--Witten invariants imply that $\bbN_d^{\rm log}(Y(D)) = \bbN^{\rm log}_{\iota(d)}(\mathrm{dP}_3(0,2))$, $\bbO_d(Y^{\rm op}(D)) = \bbO_{\iota(d)}(\mathrm{dP}^{\rm op}_3(0,2))$ under the corresponding lattice monomorphism $\pi'*: \hhh_2(Y,\bbZ) \hookrightarrow \hhh_2(\mathrm{dP}_3, \bbZ)$, meaning that it suffices to check \cref{conj:higherg} for the single case $Y(D)=\mathrm{dP}_3(0,2)$. 
The l.h.s. of \eqref{eq:logopengen} in that case was computed in \cite[Sec.~6.3.1]{Bousseau:2020fus} to be 
\beq 
\bbO_d\l(\mathrm{dP}^{\rm op}_3(0,2)\r) =
(-1)^{d_1+d_2+d_3} \frac{[d_1]_q}{d_1 [d_0]_q [d_1+d_2+d_3-d_0]_q}\qbinom{d_3}{d_0-d_1}_q 
\qbinom{d_3}{d_0-d_2}_q
\qbinom{d_0}{d_3}_q
\qbinom{d_1+d_2+d_3-d_0}{d_3}_q,
\label{eq:opendp302}
\eeq 
where we decomposed $d=d_0 (H-E_1-E_2-E_3)+d_1 E_1+d_2 E_2+d_3 E_3$ in terms of generators $\{H-E_1-E_2-E_3, E_1,E_2,E_3\}$ of
$\hhh_2(\mathrm{dP}_3,\bbZ)$, and for non-negative integers $n$, $m$ we denoted\footnote{By \cite[Prop.~2.5]{Bousseau:2020fus}, the effectiveness of $d$  implies that all arguments of the $q$-binomial expressions in \eqref{eq:opendp302} are non-negative integers.} 
\beq 
[n]_q!\coloneqq \prod_{i=1}^{n}[i]_q, \quad 
			\qbinom{n}{m}_q \coloneqq \begin{cases}
				\frac{[n]_q!}{[m]_q! \, [n-m]_q!} & 0\leq m \leq n, \\
				0 & \text{otherwise}. 
			\end{cases}
		\end{equation}
The first equality \eqref{eq:logopengen} in \cref{conj:higherg} is then a consequence of the following

\begin{prop}[=\cref{prop:NlogdP302scat,prop:Gclosedform}]
With the conventions above, we have
\begin{equation}
			\mathbb{N}^{\rm log}_d\big(\mathrm{dP}_3(0,2)\big)(\hbar) = \frac{[d_1]_q [d_2+d_3]_q}{[d_0]_q [d_1+d_2+d_3-d_0]_q}\qbinom{d_3}{d_0-d_1}_q 
\qbinom{d_3}{d_0-d_2}_q
\qbinom{d_0}{d_3}_q
\qbinom{d_1+d_2+d_3-d_0}{d_3}_q \,.
			\label{eq:NlogdP302closedform}
	\eeq 
	\label{prop:NlogdP302closedform}
\end{prop}

    Indeed, comparing \eqref{eq:opendp302} with \eqref{eq:NlogdP302closedform} we see that these generating series are related as in \eqref{eq:logopengen}.  	The second statement, equation \eqref{eq:lmovlogBPS}, then also follows from  \cref{prop:NlogdP302closedform} combined with the BPS integrality result of \cite[Thm.~8.1]{Bousseau:2020fus} for $l=2$, whose proof applies identically to this case. \\
	
	We will show \cref{prop:NlogdP302closedform} in two main steps. We will first construct a toric model for $\mathrm{dP}_3(0,2)$ in the sense of \cite{GHKlog}, and then compute the $\lambda_g$-log Gromov--Witten invariants \eqref{eq:loggw} from the corresponding quantum scattering diagram and algebra of quantum broken lines \cites{mandel2015scattering, ManRu, bousseau2018quantum, MR4048291, davison2019strong, bousseau2020strong,grafnitz:2020tropical,GRZ:2022properLG}.
 The lack of tameness is epitomised by the existence of infinite scattering when two quantum walls meet, and the resulting sum over quantum broken lines leads to the intricate-looking multi-variate generalised hypergeometric sum in \eqref{eq:NlogdP302scat}. The second step consists of proving that this sum admits a closed-form $q$-hypergeometric resummation given by \eqref{eq:NlogdP302closedform}. To our knowledge, this has not previously appeared in the literature, with the exception of the special cases $d_3=d_2=d_0$ and $d_3=d_0=d_1+d_2$ considered in \cite{Kra21}.
	%, corresponding respectively to $Y(D)=\mathrm{dP}_1(0,4)$ and $Y(D)=\bbF_0(0,4)$. 
	To prove it, we first establish a %$q$-hypergeometric 
 difference equation satisfied by a 1-parameter deformation of \eqref{eq:NlogdP302scat}, and then show inductively that the resulting recursion has a unique closed-form solution compatible with \eqref{eq:NlogdP302scat} by repeated use of Jackson's $q$-analogue of the Pfaff--Saalsch\"utz summation for the $_3\phi_2$-hypergeometric function. 

	\subsection*{Acknowledgements} 
	
	We thank P.~Bousseau and M.~van Garrel for many enlightening discussions surrounding the topic of this paper. We are particularly indebted to C.~Krattenthaler for an illuminating e-mail correspondence occurred prior to the appearance of \cite{Kra21}. He exposed us to the idea that our sought-for $q$-hypergeometric identities were too rigid to be tackled directly, whereas suitable parametric refinements might be amenable to an effective recursive strategy. Special thanks are owed to him for this insight, which is key to the arguments of \cref{sec:recrel}. We are also grateful to the anonymous referees for their valuable input.
	
	\section{Proof of \Cref{conj:higherg}}

    \subsection{Log GW invariants from the quantum tropical vertex} 
    %We will perform an explicit calculation of the log Gromov--Witten invariants of $\mathrm{dP}_3(0,2)$ using the quantum scattering diagram associated to a toric model of this geometry and performing a quantum broken line count. 
    We start off by giving a  summary of the combinatorial setup for the calculation of the higher genus log GW invariants \eqref{eq:GWexp_intro} from the associated quantum scattering diagram, referring the reader to \cite[Sec.~4.2]{Bousseau:2020fus} for a more extensive treatment.

    \subsubsection{Toric models.} 
    Two birational operations on log Calabi--Yau surfaces $Y(D)$ will feature prominently in the rest of the paper.
    \begin{itemize}
    \item If $\widetilde{Y}$ is the blow-up of $Y$ at a node of $D$ and $\widetilde{D}$ is the preimage of $D$ in $\widetilde{Y}$ we say $\widetilde{Y}(\widetilde{D})$ is a \define{corner blow-up} of $Y(D)$.
    \item In case $\widetilde{Y}$ is the blow-up of $Y$ at a smooth point in $D$ and $\widetilde{D}$ is the strict transform of $D$ in $\widetilde{Y}$ we say $\widetilde{Y}(\widetilde{D})$ is an \define{interior blow-up} of $Y(D)$.
    \end{itemize}
    
    Starting from a Looijenga pair $Y(D)$, we will seek to construct pairs $\widetilde{Y}(\widetilde{D})$ and $\overline{Y}(\overline{D})$ fitting into a diagram
    \begin{equation}
    \label{eq:toricmodel}
        \begin{tikzcd}
            & \widetilde{Y}(\widetilde{D}) \arrow{ld}[swap]{\phi} \arrow{rd}{\pi} & \\
            Y(D) & & \overline{Y}(\overline{D})
        \end{tikzcd}
    \end{equation}
    where $\phi$ is a sequence of corner blow-ups and $\pi$ is a {\it toric model}, meaning that $\overline{Y}$ is toric, $\overline{D}$ is its toric boundary and $\pi$ is a sequence of interior blow-ups. By \cite[Prop.~1.3]{GHKlog} such a diagram always exists. %\\
  %  Given a toric model for $Y(\cup_{i=1}^l D_i)$ as above 
  We will write $\rho_{D_i}$ for the generator of the ray in the fan $\overline{\Sigma}$ of $\overline{Y}$ associated to the toric prime divisor which is the push-forward along $\pi$ of the strict transform of $D_i$ under $\phi$. \\
  
  Given a toric model $\pi:\widetilde{Y}(\widetilde{D})\rightarrow\overline{Y}(\overline{D})$ as in \eqref{eq:toricmodel} we can associate a {\it consistent quantum scattering diagram} $\mf{D}$ to it, which is what we discuss next.

    \subsubsection{Quantum scattering and higher genus invariants.} \label{sec:qScatt}  The scattering diagram $\mathfrak{D}$ is defined on the lattice $N_\bbZ \cong \Z^2$ 
    of integral points of the fan $\overline{\Sigma}$ of $\overline{Y}$, as follows: assume that $\pi$ is a sequence of blow-ups of distinct smooth points $P_1,\ldots,P_s$ of $\overline{D}$ and denote $\mc{E}_1,\ldots,\mc{E}_s$  the exceptional divisors in $\widetilde{Y}$ associated to these blow-ups. Further, write $\rho_j$ for the primitive generator of the ray associated to the toric prime divisor which $P_j$ is an element of. For each $j\in\{1,\ldots,s\}$ we introduce a wall $\mf{d}_j\coloneqq \bbR\rho_j$ decorated with wall-crossing functions $f_{\mf{d}_j}\coloneqq 1+ t_j z^{-\rho_j}\in \bbC\llbracket t_1,\ldots,t_s\rrbracket [N_\bbZ]$. We will often write  $z^\rho \eqqcolon x^a y^b$ if $\rho=(a,b)$. 
    Then $\mf{D}$ is the  unique (up to equivalence) completion of the initial scattering diagram $\mf{D}_{\mr{in}}\coloneqq \{(\mf{d}_j,f_{\mf{d}_j})\}_{j\in\{1,\ldots,s\}}$ in the sense of \cites{MR2181810,GPS,MR4048291,bousseau2018quantum}. Such a completion can be found algorithmically by successively adding new rays whenever two walls meet, as we now describe. \\
    
    First of all, after perturbing the diagram $\mf{D}_{\mr{in}}$ we may assume that walls intersect in codimension at least one and that no more than two walls meet in a point. Now suppose two walls $\mf{d}^1$, $\mf{d}^2$ intersect. Denote by $-\rho^i$ a primitive integral direction of $\mf{d}^i$ and assume $f_{\mf{d}^i}=1+c_i z^{\rho^i}$. For our purpose the relevant scattering processes are:
    \begin{itemize}
        \item $\det(\rho^1,\rho^2)=\pm 1$ (\textit{simple scattering}): the algorithm tells us to add a ray $\mf{d}$ emanating from the intersection point $\mf{d}^1\cap\mf{d}^2$ in the direction $-\rho^1-\rho^2$ decorated with $1+c_1 c_2 z^{\rho^1+\rho^2}$.
        \item $\det(\rho^1,\rho^2)=\pm 2$ (\textit{infinite scattering}): the algorithm creates three types of walls. First, there is a central wall in the direction $-\rho^1-\rho^2$ decorated with a wall-crossing function whose explicit shape is not of interest in the subsequent analysis and hence omitted. Further -- and most relevant for us later --  one needs to add walls $\mf{d}_1,\ldots, \mf{d}_n,\ldots$ with slope $-(n+1)\rho^1-n\rho^2$ decorated with
        \begin{equation*}
        1+c_1^{n+1} c_2^{n} z^{(n+1)\rho^1+n\rho^2}
        \end{equation*}
        and last a collection ${}_1\mf{d},\ldots, {}_n\mf{d},\ldots$ of walls respectively having slope $-n\rho^1-(n+1)\rho^2$ and decorated with
        \begin{equation*}
        1+c_1^{n} c_2^{n+1} z^{n\rho^1+(n+1)\rho^2}\,.
        \end{equation*}
    \end{itemize}
    Adding new walls to the scattering diagram possibly creates new intersection points. Perturbing the diagram if necessary and repeating the above described process for each newly created intersection point yields a consistent scattering diagram $\mf{D}$.\\
    
    We now introduce the final combinatorial object we require for our computation of log Gromov--Witten invariants. Let $N_\bbR \coloneqq N_\bbZ \otimes_\bbZ \bbR$. Given $p\in N_\bbR$ and $m\in N_\bbZ$, a \define{quantum broken line} $\beta$ \define{with ends} $p$ \define{and} $m$ consists of
    \begin{enumerate}
        \item a continuous piece-wise straight path $\beta:(-\infty, 0]\rightarrow B \setminus \bigcup_{\mf{d}\in\mf{D}}\partial\mf{d} \cup \bigcup_{\mf{d}^1\neq\mf{d}^2} \mf{d}^1\cap\mf{d}^2$ intersecting walls transversely;
        
        \item a labelling $L_1,\ldots,L_n$ of the successive line segments with $L_n$ ending at $p$ such that each intersection point $L_i\cap L_{i+1}$ lies on a wall;
        
        \item\label{item:qblproperty3} an assignment $a_i z^{m_i}$ to each line segment $L_i$ such that
        \begin{itemize}
            \item $a_1 z^{m_1} = z^m$,
            \item if $L_i\cap L_{i+1}\subset \mf{d}$ with $f_{\mf{d}}=\sum_{r\geq0} c_r z^{r \rho_{\mf{d}}}$ and $\rho_{\mf{d}}$ chosen primitive then $a_{i+1} z^{m_{i+1}}$ is a monomial occurring in the expansion of
            \begin{equation}
            \label{eq:wallcrossingtrafo}
            a_{i} z^{m_{i}}\prod_{\ell= -\frac{1}{2}(|\det(\rho_{\mf{d}},m_{i})|-1)}^{\frac{1}{2}(|\det(\rho_{\mf{d}},m_{i})|-1)}\left(\sum_{r\geq0} c_r q^{r\ell} z^{r \rho_{\mf{d}}}\right)\,,
            \end{equation}
            \item $L_i$ is directed in direction $-m_i$.
        \end{itemize}
    \end{enumerate}
    For such a quantum broken line $\beta$ call $a_{\beta,\mr{end}}\coloneqq a_n$ the end-coefficient and write $m_{\beta,\mr{end}}\coloneqq m_n$. Moreover, we will often refer to $a_{\beta,\mr{end}} z^{m_{\beta,\mr{end}}}$ as the end-monomial and to $z^m$ as the asymptotic monomial of $\beta$. We remark that for $f_{\mf{d}}=1+c z^{\rho_{\mf{d}}}$ the product in \eqref{eq:wallcrossingtrafo} takes the form
    \begin{equation}
    \label{eq:qbinomialthm}
        \prod_{\ell= -\frac{1}{2}(|\det(\rho_{\mf{d}},m_{i})|-1)}^{\frac{1}{2}(|\det(\rho_{\mf{d}},m_{i})|-1)}\left(1 + c q^{\ell} z^{ \rho_{\mf{d}}}\right) = \sum_{k=0}^{|\det(\rho_{\mf{d}},m_{i})|} \qbinom{|\det(\rho_{\mf{d}},m_{i})|}{k}_q c^k z^{k\rho_{\mf{d}}}\,.
    \end{equation}
    Now given two lattice vectors $m_1,m_2\in N_\bbZ$, define 
    %so that the cone spanned by them 
    \begin{equation}
        \label{eq:sumqbl}
        C^{\mf{D}}_{p;\,m_1,m_2}(q) \coloneqq \sum_{\substack{\beta_1,\beta_2  \\ \mr{Ends}(\beta_i) = (p, m_i) \\ m_{\beta_1,\mr{end}} + m_{\beta_2,\mr{end}} = 0}} a_{\beta_1,\mr{end}} \, a_{\beta_2,\mr{end}}
    \end{equation}
    to be the sum of products of all end-coefficients of quantum broken lines $\beta_1$ and $\beta_2$ with asymptotic monomials $z^{m_1}$, resp.\  $z^{m_2}$, meeting in a common point $p$ and with opposite exponents of their end-monomials. This sum is a polynomial in the variables $t_j$ with coefficients in $\bbZ[q^{\pm \frac{1}{2}}]$. It turns out that the definition of $C^{\mf{D}}_{p;\,m_1,m_2}(q)$ is mostly independent of $p$. \\
    \begin{prop} \cite[Proposition 2.13 \& 2.15]{mandel2015scattering}
        The sum in \eqref{eq:sumqbl} is finite and as long as $p$ is chosen generic, $C^{\mf{D}}_{p;\,m_1,m_2}(q)$ is independent of the choice of $p$.
    \end{prop}
    Moreover, we remark that according to the same proposition \cite[Proposition 2.15]{mandel2015scattering} $C^{\mf{D}}_{p;\,m_1,m_2}(q)$ enjoys the interpretation as being the constant term in the product of two theta functions. However, most crucial for us, this quantity gives us a way to extract the higher genus, maximal contact log Gromov--Witten invariants \eqref{eq:GWexp_intro} of a Looijenga pair $Y(D)$ from the scattering diagram associated to its toric model, as per the following

    \begin{prop}[\cite{Man19,Bousseau:2020fus}]
    \label{prop:logGWfromQBL}
    Let $Y(D)$ be a 2-component Looijenga pair, i.e.\ $D=D_1+D_2$,  $\pi:\widetilde Y(\widetilde D) \to \overline{Y}(\overline{D})$ a toric model for it as in \eqref{eq:toricmodel}, and $\mathfrak{D}$ the corresponding consistent quantum scattering diagram. For $d\in\hhh_2(Y,\bbZ)$ an effective curve class, set $m_i\coloneqq (d\cdot D_i)\rho_{D_i}$. Then $\mathbb{N}^{\rm log}_d(Y(D))(\hbar)$ is the coefficient of $\prod_{j=1}^s t_j^{d \cdot \phi_* \mc{E}_j}$ in 
    \begin{equation*}
        C^{\mf{D}}_{p;\,m_1,m_2}(q)\,\bigg|_{q=\re^{\ri\hbar}}\,.
    \end{equation*}
    \end{prop}

    \subsubsection{Birational invariance.}

    Suppose now that $Y(D)$ is a 2-component quasi-tame Looijenga pair and $\pi:Y'(D')\to Y(D)$ is a sequence of interior blow-ups such that $Y'(D')$ is also quasi-tame. The following compatibility statement explains how the higher genus log-open correspondence interacts with this type of birational transformations.
    
    \begin{prop}\cite[Prop.\ 4.3 and 6.9]{Bousseau:2020fus} Let $\pi:Y'(D')\rightarrow Y(D)$ be an interior blow-up, with both $Y'(D')$ and $Y(D)$ 2-component quasi-tame Looijenga pairs. Then $\mathbb{N}^{\rm log}_d(Y(D)) = \mathbb{N}^{\rm log}_{\pi^*d}(Y'(D'))$, $     \mathbb{O}_{d}(Y^{\rm op}(D)) = \mathbb{O}_{\pi^*d}(Y'^{\rm op}(D'))$ for all $d\in\hhh_2(Y,\bbZ)$.
    \label{prop:birinv}
    \end{prop}
    
    The comparison statement of \cref{prop:birinv} for log invariants is easy to visualise in genus 0, where the invariants $N^{\rm log}_{0,d}(Y(D))$ are enumerative \cite{Man19}: since blowing up a point away from the curves does not affect the local geometry, the corresponding counts are invariant, a  property also reflected in the broken lines calculations of the scattering diagrams of \cref{sec:qScatt}. The corresponding statement for the open invariants is a non-trivial consequence of the invariance of the topological vertex under flops \cite{Iqbal:2004ne,MR2380471}.\\
    
    By the classification theorem of nef Looijenga pairs in \cite[Prop.~2.2]{Bousseau:2020fus}, any non-tame, quasi-tame Looijenga pair $Y(D)$ is obtained up to deformation as a sequence of $m\geq 0$ interior blowings-down of $\mathrm{dP}_3(0,2)$. It follows from \cref{prop:birinv} that proving \cref{conj:higherg} for $Y(D)=\mathrm{dP}_3(0,2)$ implies that the same statement {\it a fortiori} holds for any other non-tame, quasi-tame pair (and indeed any 2-component quasi-tame pair with Picard number lower than four). 
    
%\end{rmk}

	\subsection{The case of maximal Picard number}
	\label{sec:proof}
		Let us then specialise to $Y(D)=\mathrm{dP}_3(0,2)$. Throughout this section we will write $d = d_0 (H-E_1-E_2-E_3) + \sum_{i=1}^3 d_i E_i$ for a curve class $d\in\hhh_2(\mathrm{dP}_3,\Z)$. If $d\cdot D_1$ or $d\cdot D_2$ vanishes, then $\mathbb{N}^{\rm log}_d\big(\mathrm{dP}_3(0,2)\big)(\hbar)=0$. In case the intersection numbers are strictly positive, we may use the scattering diagram of $\mathrm{dP}_3(0,2)$ to compute the invariants. 
		%Remember that we chose $D_1$ in class $H-E_1$ and $D_2$ a member of $2H-E_2-E_3$.
		
	\subsubsection{Constructing the toric model.}\label{sec:tormod} First, we recall from \cite[Sec. 4.4]{Bousseau:2020fus} the construction of a toric model of $\mathbb{P}^2(D_1\cup D_2)$, with $D_1$ a line and $D_2$ a smooth conic intersecting $D_1$ in two distinct points $P_1$ and $P_2$ . Let $\mc{E}$ denote the tangent line through $P_1$ to $D_2$. In the following we will always identify $D_1$, $D_2$, $\mc{E}$, and some yet to be defined divisors $F_1$, $F_2$ with their strict transforms, resp.~push-forwards, under blow-ups, resp.~blow-downs. \\
	
	Blow up the point $P_1$, and write $F_1$ for the  exceptional divisor, and then blow up the intersection point of $F_1$ with $D_2$ and denote the exceptional curve by $F_2$.  Under these blow-ups, the strict transform of $\mc{E}$ is a $(-1)$-curve intersecting $F_2$ in one point and can therefore be blown down.
		This blow-down results in the log Calabi--Yau surface $\big(\overline{\mathbb{P}^2(D)},\overline{D}\big)$ with anti-canonical divisor $\overline{D}=D_1 \cup F_1 \cup F_2 \cup D_2$ where $F_1$ is a curve with self-intersection $-2$, $D_2$ has self-intersection $2$, and $D_1$, $F_2$ zero. Therefore, since the irreducible components of $\overline{D}$ form a necklace with the same self-intersections as the toric boundary of $\mathbb{F}_2$, we already must have $\overline{\mathbb{P}^2(D)} = \mathbb{F}_2$ with $\overline{D}$ the toric boundary by \cite[Lemma 2.10]{Fri16} and hence we have constructed a toric model for $\mathbb{P}^2(D)$. From the discussion of the previous Section, at a tropical level the fact that we blew down $\mc{E}$  amounts to adding a wall $\mathfrak{d}_{F_2}$ emanating from a {\it focus-focus singularity} on the ray corresponding to $F_2$ in the fan of $\mathbb{F}_2$.\\
	
The above construction results in a toric model for $\mathrm{dP}_3(0,2)$, as displayed in \Cref{fig:dP302}: we blow up an interior point on $D_1$ and two interior points on $D_2$ and take the proper transforms, which at the fan level leads to the addition of focus-focus singularities on the rays corresponding to $D_1$ and $D_2$ (indicated with crosses in \Cref{fig:dP302}). We denote by $\mc{E}_1$, resp.~by $\mc{E}_2$, $\mc{E}_3$ the exceptional loci that result from blowing up a point in $D_1$, resp.~two points on $D_2$.
  %The resulting picture of the fan of $\overline{\bbP^2}(\overline{D})=\bbF_2(\overline{D})$ decorated with the focus-focus singularities of the  toric model of $\mathrm{dP}_3(0,2)$ is displayed in \Cref{fig:dP302}.
		
					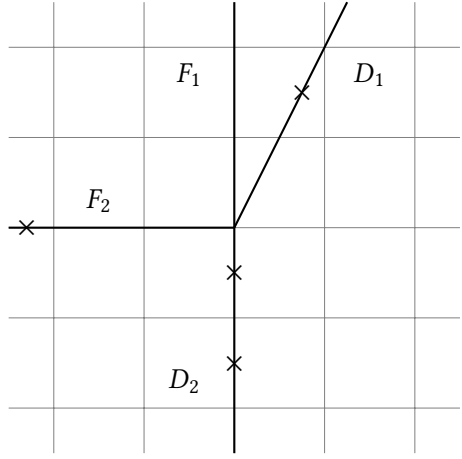
\begin{figure}[h]%
			\begin{center}%
				\begin{tikzpicture}[smooth, scale=1.2]%
					\draw[step=1cm,gray,very thin] (-2.5,-2.5) grid (2.5,2.5);
					\draw[thick] (0,0) to (-2.5,0);
					\draw[thick] (0,-2.5) to (0,2.5);
					\draw[thick] (0,0) to (1.25,2.5);
					\node at (-0.55,-1.7) {$D_2$};
					\node at (1.5,1.7) {$D_1$};
					\node at (-0.5,1.7) {$F_1$};
					\node at (-1.5,0.3) {$F_2$};
					\node at (-2.3,0) {$\times$};
					\node at (0.75,1.5) {$\times$};
					\node at (0,-0.5) {$\times$};
					\node at (0,-1.5) {$\times$};
				\end{tikzpicture}%
				\caption{The toric model of $\mathrm{dP}_3(0,2)$}%
				\label{fig:dP302}%
			\end{center}%
		\end{figure}
		
	\subsubsection{The quantum scattering diagram.}

		We now follow the construction outlined in \Cref{sec:qScatt} to derive the quantum scattering diagram $\mf{D}$ of the toric model of $\mathrm{dP}_3(0,2)$ (\Cref{fig:dP302scatt}). We shoot out walls $\mf{d}_{F_2}$, $\mf{d}_{D_1}$, $\mf{d}_{D_2,1}$, $\mf{d}_{D_2,2}$ emanating from the focus-focus singularities
			in the direction $-\rho$, where $\rho$ is the respective generator of the ray in the fan. We send the singularities to infinity and perturb the walls as indicated in \Cref{fig:dP302scatt}. From these initial walls we now want to construct a consistent scattering diagram. However, since in our subsequent analysis we will only be interested in walls with slope lying in the cone generated by $(0,-1)$ and $(1,2)$, we will restrict the discussion to such walls only. As $|\rho_{F_2}\wedge \rho_{D_1}| = 2$ there is infinite scattering in the sense of \cref{sec:qScatt} between the walls $\mf{d}_{F_2}$ and $\mf{d}_{D_1}$. This results in walls $\mf{d}_2,\mf{d}_3,\ldots$ with slope $-n \rho_{F_2} - (n-1)\rho_{D_1} = (1,-2(n-1))$ decorated with wallcrossing functions $1+t^{n} t_1^{n-1} x^{-1} y^{2(n-1)}$ where $n>1$. For conformity, let us write $\mf{d}_1\coloneqq \mf{d}_{F_2}$. Now for all $n\geq 1$ each wall $\mf{d}_n$ intersects both $\mf{d}_{D_2,1}$ and $\mf{d}_{D_2,2}$. Luckily, in this case we only have simple scattering resulting in walls with slope $(1,-2n+3)$ and wallcrossing functions $1+t^{n} t_1^{n-1} t_i x^{-1} y^{2n-3}$ where $i\in\{2,3\}$. Lastly, we notice that the wall which is the result of scattering between $\mf{d}_n$ and $\mf{d}_{D_2,1}$ intersects $\mf{d}_{D_2,2}$ thus producing a wall with slope $(1,-2n+4)$ and wallcrossing function $1+t^{n} t_1^{n-1} t_2 t_3 x^{-1} y^{2n+4}$ attached to it. Let us call this wall $\mf{d}_n^{D_2,D_2}$. The whole construction is summarised in \Cref{fig:dP302scatt}.
				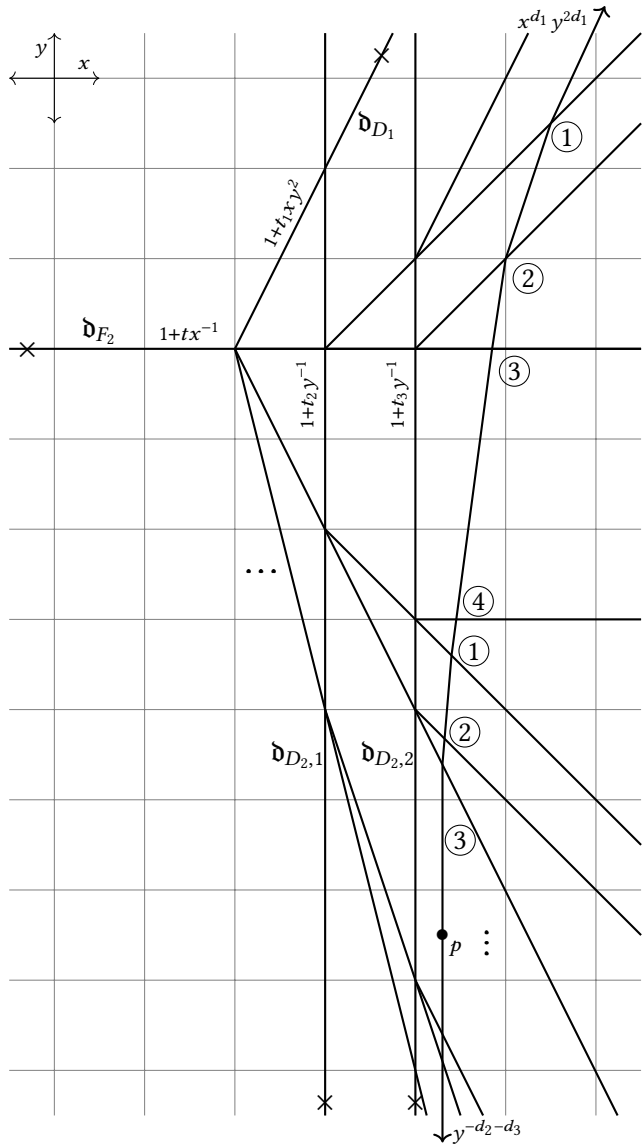
\begin{figure}[!b]
	\begin{center}
		\begin{tikzpicture}[smooth, scale=1.2]
			\draw[step=1cm,gray,very thin] (-2.5,-8.5) grid (4.5,3.5);
			\draw[<->] (-2.5,3) to (-1.5,3);
			\draw[<->] (-2,2.5) to (-2,3.5);
			\node at (-1.67,3.15) {$\scriptstyle{x}$};
			\node at (-2.15,3.3) {$\scriptstyle{y}$};
			\draw[thick] (-2.5,0) to (4.5,0);
			\draw[thick] (1,-8.5) to (1,3.5);
			\draw[thick] (2,-8.5) to (2,3.5);
			\draw[thick] (0,0) to (1.75,3.5);
			\draw[thick] (0,0) to (4.5,0);
			\draw[thick] (1,0) to (4.5,3.5);
			\draw[thick] (2,0) to (4.5,2.5);
			\draw[thick] (2,1) to (3.25,3.5);
			\draw[thick] (0,0) to (4.24,-8.5);
			\draw[thick] (1,-2) to (4.5,-5.5);
			\draw[thick] (2,-3) to (4.5,-3);
			\draw[thick] (2,-4) to (4.5,-6.5);
			\draw[thick] (0,0) to (2.125,-8.5);
			\draw[thick] (1,-4) to (2.5,-8.5);
			\draw[thick] (2,-7) to (2.75,-8.5);
			\node at (0.3,-2.5) { $\pmb{\cdots}$};
   			\node at (2.8,-6.5) {$\pmb{\vdots}$};
			\node at (1.58,2.48) {$\mathfrak{d}_{D_1}$};
			\node at (-1.5,0.2) {$\mathfrak{d}_{F_2}$};
			\node at (0.7,-4.5) {$\mathfrak{d}_{D_2,1}$};
			\node at (1.7,-4.5) {$\mathfrak{d}_{D_2,2}$};
			\node at (-2.3,0) {$\times$};
			\node at (1.625,3.25) {$\times$};
			\node at (1,-8.35) {$\times$};
			\node at (2,-8.35) {$\times$};
			\node at (-0.5,0.2) {$\sstyle{1+tx^{-1}}$};
			\node[rotate=63.43] at (0.55,1.5) {$\sstyle{1+t_1xy^2}$};
			\node[rotate=90] at (0.8,-0.5) {$\sstyle{1+t_2y^{-1}}$};
			\node[rotate=90] at (1.8,-0.5) {$\sstyle{1+t_3y^{-1}}$};
			\draw[<->,thick] (2.3,-8.8) to (2.3,-4.6) to (2.4,-3.4) to (2.85,0) to (3,1) to (3.5,2.5) to (4.1,3.8);
			\node at (3.54,3.63) {$\sstyle{x^{d_1} y^{2d_1}}$};
			\node at (2.8,-8.68) {$\sstyle{y^{-d_2-d_3}}$};
			\node at (2.3,-6.5) {$\bullet$};
			\node at (2.45,-6.65) {$\sstyle{p}$};
			\node at (3.68,2.35) {$\circleed{1}$};
			\node at (3.25,0.78) {$\circleed{2}$};
			\node at (3.1,-0.25) {$\circleed{3}$};
			\node at (2.7,-2.8) {$\circleed{4}$};
			\node at (2.65,-3.35) {$\circleed{1}$};
			\node at (2.55,-4.25) {$\circleed{2}$};
			\node at (2.5,-5.45) {$\circleed{3}$};
		\end{tikzpicture}
		\caption{The quantum scattering diagram of $\mathrm{dP}_3(0,2)$.}
		\label{fig:dP302scatt}
	\end{center}%
\end{figure}
		
		We collect the walls constructed above into 4-tuples labelled by an integer $n>0$ as depicted in \Cref{fig:walltuple}. The $n$th tuple consists of the wall $\mf{d}_n$, the walls which are the result of scattering between $\mf{d}_n$ and $\mf{d}_{D_2,1}$, resp.\ $\mf{d}_{D_2,2}$, and lastly the wall $\mf{d}_{n+1}^{D_2,D_2}$.
		
		\subsubsection{Higher genus log GW invariants.}
		
		In this section we will apply \cref{prop:logGWfromQBL} to obtain the log Gromov--Witten invariants of $\mathrm{dP}_3(0,2)$. For this we need to determine $C^{\mf{D}}_{p;\,m_1,m_2}(q)$, where $m_i\coloneqq (d\cdot D_i) \rho_{D_i}$ and $p$ is a generically chosen point. First, let us characterise all quantum broken lines which contribute to this sum.

    \begin{lem}
        \label{lem:qblcharacterisation}
        Choose $p$ in the lower right quadrant so that it lies to the right of $\mf{d}_{D_2,2}$ and that below $p$ there are only walls belonging to the $n$th 4-tuple of walls with $n\geq d\cdot D_2+2$. Then the following statements hold:
        \begin{enumerate}
            \item If $\beta_2$ is a quantum broken line with ends $(p,m_2)$ then either it is a straight line and thus $m_{\beta_2,\mr{end}}=m_2$ or $m_{\beta_2,\mr{end}}$ lies in the half open cone $\{-a_1 \rho_{D_1}-a_2\rho_{D_2} \,|\, a_1 > 0,\, a_2 \geq 0  \}$.
            
            \item If $\beta_1$ and $\beta_2$ are quantum broken lines with ends $(p,m_1)$ and $(p,m_2)$ respectively such that
            \begin{equation*}
                m_{\beta_1,\mr{end}} + m_{\beta_2,\mr{end}} = 0
            \end{equation*}
            then $\beta_2$ is a straight line and $\beta_1$ may only bend at walls to the right of $\mf{d}_{D_2,2}$ as in \Cref{fig:dP302scatt}.
        \end{enumerate}
    \end{lem}
    
    \begin{proof}
        Statement 1.~can be proven by a straightforward, but tedious case-by-case analysis. Since this proof is barely enlightening we omit it here, and only explain how 1.~implies 2. Indeed, suppose $\beta_2$ is not a straight line. Then by 1.~we must have $m_{\beta_1,\mr{end}}\in \{a_1 \rho_{D_1} + a_2\rho_{D_2} \,|\, a_1 > 0,\, {a_2 \geq 0}\}$. However, this means that $\beta_1$ only crosses walls at which it may pick up a contribution that bends it further into the the lower right quadrant. In particular, such a quantum broken line cannot have asymptotic monomial $z^{m_1}$, leading to a contradiction. Hence, $\beta_2$ must be a straight line.
    \end{proof}

    Having \cref{lem:qblcharacterisation} at hand, we are now equipped to determine the log Gromov--Witten invariants of $\mathrm{dP}_3(0,2)$ via \cref{prop:logGWfromQBL}.
    
	\begin{prop}
		Let $d$ be an effective curve class with $d\cdot D_1, d\cdot D_2>0$. Then
		\begin{equation}
			\begin{split}
				&\mathbb{N}^{\rm log}_d\big(\mathrm{dP}_3(0,2)\big)(\hbar) = \\
				&\quad \sum_{\substack{\forall (i,n)\in\{1,2,3,4\}\times\Z_{>0}: \, k_{i,n}\geq 0 \\ d_0 = \sum_{n\geq 1}\sum_{i=1}^4  (n+\delta_{i,1})  k_{i,n} \\ d_1 = \sum_{n\geq 1} \sum_{i=1}^4 k_{i,n} \\ d_0 - d_2 = \sum_{n\geq 1} (k_{1,n}+k_{4,n})  \\ d_0 - d_3 = \sum_{n\geq 1} (k_{1,n}+k_{3,n})}} ~\prod_{n\geq 1} \prod_{i=1}^{2} \qbinom{d_2+d_3 - \sum_{m\geq1} \sum_{j=1}^4 (2m-\delta_{j,3}-\delta_{j,4}) k_{j,n+m}}{k_{i,n}}_q \\[-1.2cm]
				&\hspace{13.5em} \times\qbinom{d_2+d_3 - \sum_{m\geq0} \sum_{j=1}^4 (2m+\delta_{j,1}+\delta_{j,2}) k_{j,n+m}}{k_{2+i,n}}_q.
			\end{split} 
			\label{eq:NlogdP302scat}
		\end{equation}
		\label{prop:NlogdP302scat}
	\end{prop}

	\begin{proof}

	    In order to compute $\mathbb{N}^{\rm log}_d\big(\mathrm{dP}_3(0,2)\big)$ we choose a point $p$ as specified in \cref{lem:qblcharacterisation} and consider quantum broken lines $\beta_i$ with ends $(p,z^{(d\cdot D_i) \rho_{D_i}})$, where $i\in\{1,2\}$, so that the sum of the exponents of their end-monomials at $p$ vanishes. As stated in \Cref{prop:logGWfromQBL}, we then obtain the desired log Gromov--Witten invariant by taking the product of the two end-coefficients, summing this over all such pairs of quantum broken lines and extracting the coefficient of the monomial $t^{d \cdot \phi_* \mc{E}} \prod_{i=1}^3 t_i^{d \cdot \phi_* \mc{E}_i}$.\\
		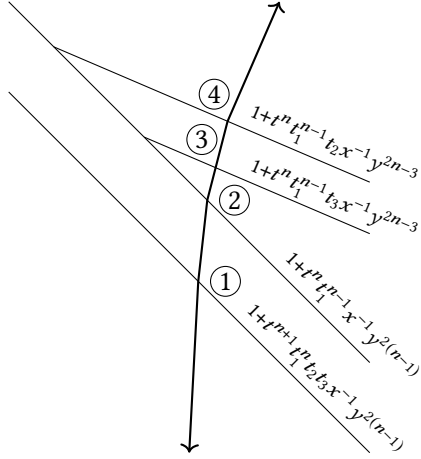
\begin{figure}[!t]%
			\begin{center}%
				\begin{tikzpicture}[smooth, scale=1.2]
					\draw (0.5,-0.5) to (4,-2);
					\node[rotate=-23.2] at (3.6,-1.6) {$\sstyle{1+t^{n}t_1^{n-1} t_2 x^{-1}y^{2n-3}}$};
					\draw (1.5,-1.5) to (4,-2.571);
					\node[rotate=-23.2] at (3.6,-2.2) {$\sstyle{1+t^{n}t_1^{n-1} t_3 x^{-1}y^{2n-3}}$};
					\draw (0,0) to (4,-4);
					\node[rotate=-45] at (3.8,-3.5) {$\sstyle{1+t^{n}t_1^{n-1} x^{-1}y^{2(n-1)}}$};
					\draw (0,-1) to (4,-5);
					\node[rotate=-45] at (3.5,-4.2) {$\sstyle{1+t^{n+1}t_1^{n} t_2 t_3 x^{-1}y^{2(n-1)}}$};
					\node at (2.4,-3.1) {$\circleed{1}$};
					\node at (2.5,-2.2) {$\circleed{2}$};
					\node at (2.14,-1.525) {$\circleed{3}$};
					\node at (2.275,-1.025) {$\circleed{4}$};
					\draw[<->,thick] (2,-5) to (2.1,-3.1) to (2.2,-2.2) to (2.425,-1.325) to (3,0);
				\end{tikzpicture}%
				\caption{The $n$th 4-tuple of walls.}
				\label{fig:walltuple}
			\end{center}%
		\end{figure}%
Let us then analyse a quantum broken line $\beta_1$ coming from the direction $D_1$ with asymptotic monomial $z^{(d \cdot D_1) \rho_{D_1}} =x^{d_1} y^{2d_1}$ ending at $p$ which only bends at walls to the right of $\mf{d}_{D_2,2}$ as illustrated in \cref{fig:dP302scatt}. We claim that its end-monomial is of the form
\begin{equation}
\begin{split}
\label{eq:endmonbeta1}
    a_{\beta_1,\mr{end}} z^{m_{\beta_1,\mr{end}}} = &\prod_{n=1}^N \prod_{i=1}^2  \qbinom{2n d_1 - \sum_{l=1}^n \sum_{j=1}^4 (2(n-l)-\delta_{j,3}-\delta_{j,4})k_{j,l} }{k_{i,n}}_q   \\
    &\hspace{2em}\times \qbinom{(2n-1) d_1 - \sum_{l=1}^{n-1} \sum_{j=1}^4 (2(n-l)-\delta_{j,1}-\delta_{j,2}) k_{j,l} }{k_{2+i,n}}_q \\
    &\times t^{\sum_{n=1}^N \sum_{j=1}^4 (n+\delta_{j,1})k_{j,n}} t_1^{\sum_{n=1}^N \sum_{j=1}^4 (n-1+\delta_{j,1})k_{j,n}} t_2^{\sum_{n=1}^N (k_{1,n} + k_{4,n})} t_3^{\sum_{n=1}^N (k_{1,n} + k_{3,n})} \\
    &\times x^{d_1 - \sum_{n=1}^N \sum_{j=1}^4 k_{j,n}} y^{2 d_1  + \sum_{n=1}^N (2n-2) (k_{1,n} + k_{2,n}) + \sum_{n=1}^N (2n-3) (k_{3,n}+k_{4,n})}
\end{split}
\end{equation}
for some $k_{j,n}\geq 0$ where we set $N=d\cdot D_2+1$. Indeed, suppose that after passing the first $(n-1)$ 4-tuples of walls $\beta_2$ carries the monomial $a_{n-1} x^{b_{n-1}} y^{c_{n-1}}$. Then crossing wall $\circleed{4}$ of the $n$th 4-tuple displayed in \Cref{fig:walltuple} we see that by the defining property \labelcref{item:qblproperty3} of a quantum broken line and formula \eqref{eq:qbinomialthm} the monomial carried by the quantum broken line must now be
\begin{equation*}
    a_{n-1} x^{b_{n-1}} y^{c_{n-1}} \stackrel{\circleed{4}}{\longmapsto} a_{n-1} \qbinom{(2n-3)b_{n-1}+c_{n-1}}{k_{4,n}}_q t^{n k_{4,n}} t_1^{(n-1)k_{4,n}} t_2^{k_{4,n}} x^{b_{n-1}-k_{4,n}} y^{c_{n-1}+(2n-3) k_{4,n}}
\end{equation*}
for some $k_{4,n}\geq 0$. Similarly, after passing the next three walls, which together with $\circleed{4}$ form the $n$th 4-tuple of walls, the quantum broken line carries the monomial
\begin{equation}
\begin{split}
    a_{n-1} & \qbinom{(2n-3)b_{n-1}+c_{n-1}}{k_{4,n}}_q t^{n k_{4,n}} t_1^{(n-1)k_{4,n}} t_2^{k_{4,n}} x^{b_{n-1}-k_{4,n}} y^{c_{n-1}+(2n-3) k_{4,n}} \\[1.5em]
    & \stackrel{\circleed{3}}{\longmapsto} a_{n-1} \qbinom{(2n-3)b_{n-1}+c_{n-1}}{k_{3,n}}_q \qbinom{(2n-3)b_{n-1}+c_{n-1}}{k_{4,n}}_q \\
    & \hspace{3em} \times t^{n (k_{3,n}+k_{4,n})} t_1^{(n-1)(k_{3,n}+k_{4,n})} t_2^{k_{4,n}} t_3^{k_{3,n}}    x^{b_{n-1}-k_{3,n}-k_{4,n}} y^{c_{n-1}+(2n-3) (k_{3,n}+k_{4,n})} \\[1.5em]
    & \stackrel{\circleed{2}}{\longmapsto} a_{n-1} \qbinom{(2n-2)b_{n-1}+c_{n-1} + (k_{3,n}+k_{4,n}) }{k_{2,n}}_q  \\
    & \hspace{3em} \times \qbinom{(2n-3)b_{n-1}+c_{n-1}}{k_{3,n}}_q \qbinom{(2n-3)b_{n-1}+c_{n-1}}{k_{4,n}}_q \\
    & \hspace{3em} \times t^{\sum_{j=2}^4 n k_{j,n}} t_1^{\sum_{j=2}^4(n-1) k_{j,n}} t_2^{k_{4,n}} t_3^{k_{3,n}}    x^{b_{n-1}-\sum_{j=2}^4 k_{j,n}} y^{c_{n-1}+(2n-2) k_{2,n} + (2n-3) (k_{3,n}+k_{4,n})} \\[1.5em]
    & \stackrel{\circleed{1}}{\longmapsto} a_{n-1} \qbinom{(2n-2)b_{n-1}+c_{n-1} + (k_{3,n}+k_{4,n}) }{k_{1,n}}_q \qbinom{(2n-2)b_{n-1}+c_{n-1} + (k_{3,n}+k_{4,n}) }{k_{2,n}}_q  \\
    & \hspace{3em} \times \qbinom{(2n-3)b_{n-1}+c_{n-1}}{k_{3,n}}_q \qbinom{(2n-3)b_{n-1}+c_{n-1}}{k_{4,n}}_q \\
    & \hspace{3em} \times t^{\sum_{j=1}^4 (n+\delta_{j,1}) k_{j,n}} t_1^{\sum_{j=1}^4(n-1 + \delta_{j,1}) k_{j,n}} t_2^{k_{1,n}+k_{4,n}} t_3^{k_{1,n}+k_{3,n}}  \\
    & \hspace{3em} \times  x^{b_{n-1}-\sum_{j=1}^4 k_{j,n}} y^{c_{n-1}+(2n-2) (k_{1,n} + k_{2,n}) + (2n-3) (k_{3,n}+k_{4,n})}\,. \label{eq:result_nth_4tuple}
\end{split}
\end{equation}
Since we know that $(b_0,c_0) = m_1 = (d_1,2d_1)$, by induction we can deduce that
\begin{equation*}
    b_n = d_1 - \sum_{l=1}^n \sum_{j=1}^4 k_{j,l} \,, \qquad c_n = 2d_1 + \sum_{l=1}^n (2l-2)(k_{1,l} + k_{2,l}) + \sum_{l=1}^n (2l-3)(k_{3,l} + k_{4,l})\,.
\end{equation*}
Substituting this into the binomial coefficients in the last line of \eqref{eq:result_nth_4tuple} and noting that
\begin{equation*}
    (2n-3)b_{n-1}+c_{n-1} = (2n-1) d_1 - \sum_{l=1}^{n-1} (2(n-l)-1) (k_{1,l} + k_{2,l}) - \sum_{l=1}^{n-1} 2(n-l) (k_{3,l} + k_{4,l})
\end{equation*}
and
\begin{equation*}
    (2n-2)b_{n-1}+c_{n-1} + (k_{3,n}+k_{4,n}) =  2n d_1 - \sum_{l=1}^n 2(n-l) (k_{1,l} + k_{2,l}) - \sum_{l=1}^n (2(n-l)-1) (k_{3,l} + k_{4,l})
\end{equation*}
we indeed see that the end-monomial $a_{\beta_1,\mr{end}}z^{m_{\beta_1,\mr{end}}}$ is of the form \eqref{eq:endmonbeta1}.\\
Now by definition, $C^{\mf{D}}_{p;\,m_1,m_2}(q)$ is the sum of  $a_{\beta_1,\mr{end}}a_{\beta_2,\mr{end}}$ over all quantum broken lines $\beta_i$, $i\in\{1,2\}$, with ends $(p,m_i)$ such that
\begin{equation}
    \label{eq:endexpcondition}
    m_{\beta_1,\mr{end}} + m_{\beta_2,\mr{end}} = 0\,.
\end{equation}
Since by \Cref{lem:qblcharacterisation} $\beta_2$ must be a straight line, we have $m_{\beta_2,\mr{end}} = (d\cdot D_2)\rho_{D_2} = (0,-d_2-d_3)$ and thus \eqref{eq:endexpcondition} translates into the conditions
		\begin{align}
			d_1 &= \sum_{n=1}^{N} \sum_{i=1}^4 k_{i,n}, \label{eq:cond11}\\
			d_2 + d_3 &= \sum_{n=1}^{N} \big(2n(k_{1,n}+k_{2,n}) + (2n-1)(k_{3,n}+k_{4,n})\big)\,. \label{eq:cond12}
		\end{align}
Hence, plugging $a_{\beta_2,\mr{end}}=1$ and \eqref{eq:endmonbeta1} into the defining expression \eqref{eq:sumqbl} for $C^{\mf{D}}_{p;\,m_1,m_2}(q)$ we get
\begin{equation*}
    \begin{split}
    & C^{\mf{D}}_{p;\,m_1,m_2}(q) = \\
    & \quad \sum_{\substack{\forall (i,n)\in\{1,2,3,4\}\times\{1,\ldots,N\}: \, k_{i,n}\geq 0 \\ d_1 = \sum_{n = 1}^N \sum_{i=1}^4 k_{i,n} \\ d_2 + d_3 = \sum_{n=1}^{N} \sum_{i=1}^4 (2n-\delta_{i,3}-\delta_{i,4})k_{i,n}}}  \prod_{n=1}^N \prod_{i=1}^2  \qbinom{2n d_1 - \sum_{l=1}^n \sum_{j=1}^4 (2(n-l)-\delta_{j,3}-\delta_{j,4})k_{j,l} }{k_{i,n}}_q   \\[-2em]
    &\quad\hspace{15.2em}\times \qbinom{(2n-1) d_1 - \sum_{l=1}^{n-1} \sum_{j=1}^4 (2(n-l)-\delta_{j,1}-\delta_{j,2}) k_{j,l} }{k_{2+i,n}}_q \\
    &\quad\hspace{13em}\times t^{\sum_{n=1}^N \sum_{j=1}^4 (n+\delta_{j,1})k_{j,n}} t_1^{\sum_{n=1}^N \sum_{j=1}^4 (n-1+\delta_{j,1})k_{j,n}} \\
    &\quad\hspace{13em} \times t_2^{\sum_{n=1}^N (k_{1,n} + k_{4,n})} t_3^{\sum_{n=1}^N (k_{1,n} + k_{3,n})}\,.
    \end{split}
\end{equation*}
Using the sum conditions \eqref{eq:cond11} and \eqref{eq:cond12} we can simplify the arguments of the $q$-binomials to bring the sum into the form
\begin{equation}
    \label{eq:qblsumresult}
    \begin{split}
    & C^{\mf{D}}_{p;\,m_1,m_2}(q) = \\
    & \quad \sum_{\substack{\forall (i,n)\in\{1,2,3,4\}\times\{1,\ldots,N\}: \, k_{i,n}\geq 0 \\ d_1 = \sum_{n = 1}^N \sum_{i=1}^4 k_{i,n} \\ d_2 + d_3 = \sum_{n=1}^{N} \sum_{i=1}^4 (2n-\delta_{i,3}-\delta_{i,4})k_{i,n}}}  \prod_{n=1}^N \prod_{i=1}^2  \qbinom{d_2 + d_3 - \sum_{m=1}^{N-n} \sum_{j=1}^4 (2m-\delta_{j,3}-\delta_{j,4})k_{j,n+m} }{k_{i,n}}_q   \\[-2em]
    &\quad\hspace{15.2em}\times \qbinom{d_2 + d_3 - \sum_{m=0}^{N-n} \sum_{j=1}^4 (2m-\delta_{j,1}-\delta_{j,2}) k_{j,n+m} }{k_{2+i,n}}_q \\
    &\quad\hspace{13em}\times t^{\sum_{n=1}^N \sum_{j=1}^4 (n+\delta_{j,1})k_{j,n}} t_1^{\sum_{n=1}^N \sum_{j=1}^4 (n-1+\delta_{j,1})k_{j,n}} \\
    &\quad\hspace{13em} \times t_2^{\sum_{n=1}^N (k_{1,n} + k_{4,n})} t_3^{\sum_{n=1}^N (k_{1,n} + k_{3,n})}\,.
    \end{split}
\end{equation}
We can now apply \Cref{prop:logGWfromQBL} which states that $\mathbb{N}^{\rm log}_d\big(\mathrm{dP}_3(0,2)\big)$ is the coefficient of
		\begin{equation}
			\label{eq:cond2monomial}
		    t^{d \cdot \phi_* \mc{E}} \prod_{i=1}^3 t_i^{d \cdot \phi_* \mc{E}_i} = t^{d_0} t_1^{d_0-d_1} t_2^{d_0-d_2} t_3^{d_0-d_3}\,.
		\end{equation}
  in \eqref{eq:qblsumresult}. Here, $\phi$ is the sequence of corner blow-ups $\widetilde{\mathrm{dP}_3}(\widetilde{D})\rightarrow \mathrm{dP}_3(D)$ and $\mc{E},\mc{E}_1,\mc{E}_2,\mc{E}_3$ are the exceptional curves we introduced in \Cref{sec:tormod}. The above identity follows from the fact that $[\phi_* \mc{E}]=H$ and $[\phi_* \mc{E}_i]=E_i$. Now picking the coefficient of \eqref{eq:cond2monomial} in  \eqref{eq:qblsumresult} we obtain that $\mathbb{N}^{\rm log}_d\big(\mathrm{dP}_3(0,2)\big)$ is the sum of 
		\begin{equation}
			\begin{split}
            \label{eq:summandNlog}
				& \prod_{n\geq 1} \prod_{i=1}^{2} \qbinom{d_2 + d_3 - \sum_{m = 1}^{N-n} \big(2m (k_{1,n+m} + k_{2,n+m}) + (2m-1)(k_{3,n+m} + k_{4,n+m})\big)}{k_{i,n}}_q \\
				& \hspace{0.78cm} \times\qbinom{d_2 + d_3 - \sum_{m = 0}^{N-n}\big((2m+1)(k_{1,n+m} + k_{2,n+m}) + 2m(k_{3,n+m} + k_{4,n+m})\big)}{k_{2+i,n}}_q
			\end{split}
		\end{equation}
  over $k_{j,n}\geq 0$, $j\in\{1,\ldots,4\}$ and $n\in \{ 1,\ldots,N\}$ subject to the conditions \eqref{eq:cond11}, \eqref{eq:cond12} and
		\begin{align}
			d_0 &= \sum_{n = 1}^N \bigg(k_{1,n} + n \sum_{i=1}^4 k_{i,n} \bigg)\,, \label{eq:cond21}\\
			d_0 - d_1 &= \sum_{n = 1}^N \bigg(k_{1,n} + (n-1) \sum_{i=1}^4 k_{i,n} \bigg)\,, \label{eq:cond22}\\
			d_0 - d_2 &= \sum_{n = 1}^N (k_{1,n} + k_{4,n})\,, \label{eq:cond23}\\
			d_0 - d_3 &= \sum_{n = 1}^N (k_{1,n} + k_{3,n}) \,.\label{eq:cond24}
		\end{align}
	Notice here that \eqref{eq:cond21}, \eqref{eq:cond23}, and \eqref{eq:cond24} together imply \eqref{eq:cond12}, and that \eqref{eq:cond22} follows from subtracting \eqref{eq:cond11} from \eqref{eq:cond21}. Hence, it is sufficient to impose conditions \eqref{eq:cond11}, \eqref{eq:cond21}, \eqref{eq:cond23}, and \eqref{eq:cond24} only. Note that these constraints are exactly the ones appearing in the sum in \eqref{eq:NlogdP302scat}. Moreover, \eqref{eq:summandNlog} exactly matches the summand on the right-hand side of \eqref{eq:NlogdP302scat} and hence the claimed formula is proven.

	\end{proof}
	
		\begin{rmk}
		It is easy to convince oneself that there are actually only finitely many summands contributing to \eqref{eq:NlogdP302scat}, due to the finite number of choices $(k_{i,n})$ satisfying the summation conditions. Moreover, the product in each of these summands is well-defined since the first sum condition forces $k_{i,n}=0$ for all $n>d_0$. Thus, only a finite number of binomials can be different from one and therefore the whole expression becomes well-defined.
	\end{rmk}
 
			\newcommand{\Gfct}{G}
	\newcommand{\vara}{a}
	\newcommand{\varb}{b}
	\newcommand{\varc}{c}
	\newcommand{\vard}{d}
	\newcommand{\vare}{e}

	%Now clearly our goal is 
	To prove that the right-hand side of \eqref{eq:NlogdP302scat} returns \eqref{eq:NlogdP302closedform}, it will be helpful to consider % However, instead of proving this identity directly we will rather prove a formula for 
	a 1-parameter deformation of  \eqref{eq:NlogdP302scat}, as follows.
 	For non-negative integers $\vara,\varb,\varc,\vard,\vare$, we write
		\begin{equation}
			\label{eq:Gdef}
			\begin{split}
				&\Gfct(\vara,\varb,\varc,\vard,\vare) \coloneqq \\
				&\quad \sum_{\substack{\forall (i,n)\in\{1,2,3,4\}\times\Z_{>0}: \, k_{i,n}\geq 0 \\ \vara = \sum_{n\geq 1} \sum_{i=1}^4 (n+\delta_{i,1}) k_{i,n} \\ \varb = \sum_{n\geq 1} \sum_{i=1}^4  k_{i,n} \\ \varc = \sum_{n\geq 1} (k_{1,n}+k_{4,n})  \\ \vard = \sum_{n\geq 1} (k_{1,n}+k_{3,n})}} ~ \prod_{n\geq 1} \prod_{i=1}^{2} \qbinom{\vare - \sum_{m \geq 1} \sum_{j=1}^4 (2m-\delta_{j,3}-\delta_{j,4}) k_{j,n+m}}{k_{i,n}}_q \\[-1.1cm]
				&\hspace{13.7em} \times\qbinom{\vare - \sum_{m \geq 0}  \sum_{j=1}^4 (2m+\delta_{j,1}+\delta_{j,2}) k_{j,n+m}}{k_{2+i,n}}_q.
			\end{split} 
		\end{equation}

    By \eqref{eq:NlogdP302scat}, our sought-for log~Gromov--Witten generating function is obtained from \eqref{eq:Gdef} via the restriction

	\begin{equation}
		\mathbb{N}^{\rm log}_d\big(\mathrm{dP}_3(0,2)\big) = \Gfct(d_0,d_1,d_0-d_2,d_0-d_3,d_2+d_3)\,.
		\label{eq:Nlogred}
	\end{equation}
 
   We claim that \eqref{eq:Gdef} has a simple closed-form summation, as follows.

  	\begin{prop}

    For all $a,b,c,d,e \in \bbZ_{\geq 0}$ we have that
		\begin{equation}
			\label{eq:Gclosedform}
			\begin{split}
				\Gfct(\vara,\varb,\varc,\vard,\vare) = &\qbinom{b-\vara+e}{b-c}_q \qbinom{c-\vara+d+e}{c}_q  \qbinom{\vara-c}{d}_q \qbinom{\vara-d}{b-d}_q \\
				& \quad - \qbinom{b-\vara+e-1}{b-c}_q \qbinom{c-\vara+d+e-1}{c}_q  \qbinom{\vara-c-1}{d}_q \qbinom{\vara-d-1}{b-d}_q.
			\end{split}
		\end{equation}
In particular, from \eqref{eq:Nlogred}, it follows that
\beq  
\mathbb{N}^{\rm log}_d\big(\mathrm{dP}_3(0,2)\big)(\hbar) = \frac{[d_1]_q [d_2+d_3]_q}{[d_0]_q [d_1+d_2+d_3-d_0]_q}\qbinom{d_3}{d_0-d_1}_q 
\qbinom{d_3}{d_0-d_2}_q
\qbinom{d_0}{d_3}_q
\qbinom{d_1+d_2+d_3-d_0}{d_3}_q \,.
\eeq
		\label{prop:Gclosedform}
	\end{prop}

%\begin{rmk}
%The statement is trivially verified 
%    whenever one of the integers $\vara,\varb,\varc,\vard$ is strictly negative, as in this case the right-hand sides of both \eqref{eq:Gdef} and \eqref{eq:Gclosedform} vanish. 	
%\end{rmk}

 We will give an inductive proof of \cref{prop:Gclosedform} in the next Section by seeking a suitable 
 %$q$-hypergeometric difference equation 
 recursive relation in the parameter $a$, broadly following the lead of \cite{Kra21}. The proof is composed of three main steps:
 \ben
 \item We first establish \eqref{eq:Gclosedform} for the base cases $a=0$ and $b=0$ by a direct analysis of the quantum broken line sum in \eqref{eq:Gdef}, which in these cases is either vanishing, or reduces to a single summand (\cref{lem:Gab0}).
 \item We then establish, for $a,b>0$, a difference equation satisfied by $G(a,b,c,d,e)$ in the parameters $a$, $b$, $c$ and $d$. This equation recursively and uniquely determines $G(a,b,c,d,e)$ from the knowledge of the base case $G(a',b',c',d',e)$ with $0\leq a'<a$ (\cref{lem:Grec}).
 \item We finally check that the r.h.s. of \eqref{eq:Gclosedform} is also a solution of the difference equation, and conclude by uniqueness (\cref{lem:Gtilderec}). The proof is fundamentally based on a classical $q$-hypergeometric summation result in the form of the $q$-Pfaff--Saalsch\"utz formula (\cref{thm:qPS}).
 \een 

	\subsubsection{Proof of \cref{prop:Gclosedform}.}
	\label{sec:recrel}

 We start by considering the base cases $\vara=0$ and $\varb=0$ in the induction procedure.
 
 \begin{lem}
     		We have that
\beq    
\label{eq:Gab0}
       \Gfct(0,\varb,\varc,\vard,\vare) =\delta_{\varb,0}\delta_{\varc,0}\delta_{\vard,0}\,,  \quad
       \Gfct(\vara,0,\varc,\vard,\vare)= \delta_{\vara,0}\delta_{\varc,0}\delta_{\vard,0}\,.   
\eeq
In particular, \cref{prop:Gclosedform} holds when $\vara=0$ or $\varb=0$.
\label{lem:Gab0}
 \end{lem}

 \begin{proof}
Both equalities are implied by the conditions on the range of summation of \eqref{eq:Gdef}: setting $\vara=0$ or $\varb=0$ implies that $k_{i,n}=0$ for all $(i,n) \in \{1,2,3,4\} \times \bbZ_{>0}$. In this case the r.h.s. of \eqref{eq:Gdef} is equal to one for $c=d=0$, while if either $c\neq 0$ or $d \neq 0$ it vanishes being an empty sum.

To see that \eqref{eq:Gab0} agrees with \eqref{eq:Gclosedform}, note that when $\vara=0$ the second summand in \eqref{eq:Gclosedform} must always vanish, while the first one is non-zero only if $\varb=\varc=\vard=0$, in which case it is equal to $1$. Likewise, when $\varb=0$, the right-hand side of \eqref{eq:Gclosedform} can only be non-zero if $\varc=\vard=0$: this is equal to $1$ for $\vara=0$ by the previous analysis, and for $\vara>0$ the contributions from the two summands cancel each other, since
		\begin{equation*}
			\begin{split}
				&\qbinom{\vare-\vara}{-\varc}_q \qbinom{\varc-\vara+\vard+\vare}{\varc}_q  \qbinom{\vara-\varc}{\vard}_q \qbinom{\vara-\vard}{-\vard}_q \\
				&\quad - \qbinom{\vare-\vara-1}{-\varc}_q \qbinom{\varc-\vara+\vard+\vare-1}{\varc}_q  \qbinom{\vara-\varc-1}{\vard}_q \qbinom{\vara-\vard-1}{-\vard}_q %\\[0.3em]
				%&
				= \delta_{\varc,0} \, \delta_{\vard,0} - \delta_{\varc,0} \, \delta_{\vard,0} = 0\,.
			\end{split}
		\end{equation*}
 \end{proof}

\begin{lem}
For $\vara,\varb \in \bbZ_{>0}$, $\varc,\vard, \vare \in \bbZ_{\geq 0}$, 		$\Gfct(\vara,\varb,\varc,\vard,\vare)$ satisfies the following recursion in $(a,b,c,d)$:
		\begin{equation}
			\begin{split}
				&\Gfct(\vara,\varb,\varc,\vard,\vare) = \\
				&\quad \sum_{k_{1,1},k_{2,1},k_{3,1},k_{4,1}\geq 0} \prod_{i=1}^{2}\qbinom{\vare-2\vara+2\varb+\varc+\vard-k_{3,1}-k_{4,1}}{k_{i,1}}_q \qbinom{\vare-2\vara+\varb+\varc+\vard}{k_{i+2,1}}_q  \\[0.3cm]
				& \hspace{3cm} \times \Gfct(\vara-\varb-k_{1,1} ,\varb-k_{1,1}-k_{2,1}-k_{3,1}-k_{4,1} ,\varc-k_{1,1}-k_{4,1} ,\vard-k_{1,1}-k_{3,1} ,\vare)\,.
			\end{split}
			\label{eq:Grec}
		\end{equation}
    
    \label{lem:Grec}
    
\end{lem}
	
	\begin{proof}

	    Looking at the defining equation \eqref{eq:Gdef} for $G(a,b,c,d,e)$, we see that the sum conditions imply that
		\begin{equation*}
			\sum_{m\geq 1} \big(2m (k_{1,1+m} + k_{2,1+m}) + (2m-1) (k_{3,1+m} + k_{4,1+m})\big) = 2\vara - 2\varb - \varc - \vard + k_{3,1} + k_{4,1}
		\end{equation*}
		and
		\begin{equation*}
			\sum_{m\geq 0} \big((2m+1) (k_{1,1+m} + k_{2,1+m}) + 2m (k_{3,1+m} + k_{4,1+m})\big) = 2\vara - \varb - \varc - \vard\,.
		\end{equation*}
		Hence, they allow us to rewrite the four $q$-binomial coefficients that correspond to the factor $n=1$ in the product over $n\geq 1$ in \eqref{eq:Gdef} as
		\begin{equation*}
		\begin{split}
		    \prod_{i=1}^{2} &\qbinom{\vare - \sum_{m \geq 1} \big(2m (k_{1,1+m} + k_{2,1+m}) + (2m-1) (k_{3,1+m} + k_{4,1+m})\big)}{k_{i,n}}_q \\
			\times&\qbinom{\vare - \sum_{m \geq 0} \big((2m+1) (k_{1,1+m} + k_{2,1+m}) + 2m (k_{3,1+m} + k_{4,1+m})\big)}{k_{2+i,n}}_q\\
			&\hspace{7em} = \prod_{i=1}^{2}\qbinom{\vare-2\vara+2\varb+\varc+\vard-k_{3,1}-k_{4,1}}{k_{i,1}}_q \qbinom{\vare-2\vara+\varb+\varc+\vard}{k_{i+2,1}}_q \,.
		\end{split}
		\end{equation*}
		This factor is independent of all $k_{i,n}$ with $n>1$ and hence in \eqref{eq:Gdef} we can factor it out of the sum  over these integers which gives us
		\begin{equation}
		\label{eq:Gexpansion}
			\begin{split}
				&\Gfct(\vara,\varb,\varc,\vard,\vare) = \\
				&\quad \sum_{k_{1,1},k_{2,1},k_{3,1},k_{4,1}\geq 0} \prod_{i=1}^{2}\qbinom{\vare-2\vara+2\varb+\varc+\vard-k_{3,1}-k_{4,1}}{k_{i,1}}_q \qbinom{\vare-2\vara+\varb+\varc+\vard}{k_{i+2,1}}_q  \\[0.3cm]
				& \hspace{4em} \times\!\!\!\!\!\!\sum_{\substack{\forall (i,n)\in\{1,2,3,4\}\times\Z_{>1}: \, k_{i,n} \geq 0 \\ \vara = \sum_{n\geq 1} \sum_{i=1}^4 (n+\delta_{i,1}) k_{i,n} \\ \varb = \sum_{n\geq 1} \sum_{i=1}^4  k_{i,n} \\ \varc = \sum_{n\geq 1} (k_{1,n}+k_{4,n})  \\ \vard = \sum_{n\geq 1} (k_{1,n}+k_{3,n})}} ~ \prod_{n\geq 2} \prod_{i=1}^{2} \qbinom{\vare - \sum_{m \geq 1} \sum_{j=1}^4 (2m-\delta_{j,3}-\delta_{j,4}) k_{j,n+m}}{k_{i,n}}_q \\[-1.1cm]
				&\hspace{16.8em} \times\qbinom{\vare - \sum_{m \geq 0}  \sum_{j=1}^4 (2m+\delta_{j,1}+\delta_{j,2}) k_{j,n+m}}{k_{2+i,n}}_q.
			\end{split} 
		\end{equation}
		Now using \eqref{eq:Gdef} again in order to explicitly write out
		\begin{equation}
		\label{eq:Gshiftedexpanded}
		    \begin{split}
		        &\Gfct(\vara-\varb-k_{1,1} ,\varb-k_{1,1}-k_{2,1}-k_{3,1}-k_{4,1} ,\varc-k_{1,1}-k_{4,1} ,\vard-k_{1,1}-k_{3,1} ,\vare) = \\
		        &\quad \sum_{\substack{\forall (i,n)\in\{1,2,3,4\}\times\Z_{>1}: \, k_{i,n}\geq 0 \\ \vara-\varb-k_{1,1} = \sum_{n\geq 2} \sum_{i=1}^4 (n-1+\delta_{i,1}) k_{i,n} \\ \varb-k_{1,1}-k_{2,1}-k_{3,1}-k_{4,1} = \sum_{n\geq 2} \sum_{i=1}^4  k_{i,n} \\ \varc-k_{1,1}-k_{4,1} = \sum_{n\geq 2} (k_{1,n}+k_{4,n})  \\ \vard-k_{1,1}-k_{3,1} = \sum_{n\geq 2} (k_{1,n}+k_{3,n})}} ~ \prod_{n\geq 2} \prod_{i=1}^{2} \qbinom{\vare - \sum_{m \geq 1} \sum_{j=1}^4 (2m-\delta_{j,3}-\delta_{j,4}) k_{j,n+m}}{k_{i,n}}_q \\[-1.1cm]
				&\hspace{17.7em} \times\qbinom{\vare - \sum_{m \geq 0}  \sum_{j=1}^4 (2m+\delta_{j,1}+\delta_{j,2}) k_{j,n+m} }{k_{2+i,n}}_q.
		    \end{split}
		\end{equation}
		we notice that the sum conditions in \eqref{eq:Gshiftedexpanded} are actually equivalent to the ones in the third line of \eqref{eq:Gexpansion}. Hence, we can identify \eqref{eq:Gshiftedexpanded} with line three and four of \eqref{eq:Gexpansion} and so we arrive at the recursion formula \eqref{eq:Grec}.

		\end{proof}

    Define now
		\begin{equation}
			\tilde{\Gfct}(\vara,\varb,\varc,\vard,\vare)\coloneqq \qbinom{\varb-\vara+\vare}{\varb-\varc}_q \qbinom{\varc-\vara+\vard+\vare}{\varc}_q  \qbinom{\vara-\varc}{\vard}_q \qbinom{\vara-\vard}{\varb-\vard}_q\,,
			\label{eq:Gtildedef}
		\end{equation}
  		so that the right-hand side of  
		\eqref{eq:Gclosedform} equates to
		%\beq 
%		{\Gfct}(\vara,\varb,\varc,\vard,\vare) 
		$\tilde{\Gfct}(\vara,\varb,\varc,\vard,\vare) - \tilde{\Gfct}(\vara-1,\varb,\varc,\vard,\vare-2)$.
		%\eeq  
%		In the proof of the next \namecref{prop:Gclosedform} 
  We shall make extensive use of Jackson's $q$-analogue of the Pfaff--Saalsch\"{u}tz summation in the form it is stated in   \cite[Equation (1q)]{ZeilSaal}.
		\begin{lem}[The $q$-Pfaff--Saalsch\"utz Theorem, \cite{ZeilSaal}]
		\label{thm:qPS}
		 For integers $A,B,C,D\geq 0$ we have
		\begin{equation}
			\sum_{k\geq 0} \frac{[A+B+C+D-k]_q!}{[k]_q!\, [A-k]_q!\, [B-k]_q!\, [C-k]_q!\, [D+k]_q!} = \qbinom{A+B+D}{B}_q \qbinom{A+C+D}{A}_q \qbinom{B+C+D}{C}_q.
			\label{eq:qPS}
		\end{equation}
		\end{lem}

\begin{prop}
Let $a,b \in \bbZ_{>0}$ and $c,d,e \in \bbZ_{\geq 0}$. 		Then  $\tilde{\Gfct}(\vara,\varb,\varc,\vard,\vare)$ satisfies the same recursion \eqref{eq:Grec} as $\Gfct(\vara,\varb,\varc,\vard,\vare)$, i.e.  
		\begin{equation}
			\begin{split}
				&\tilde{\Gfct}(\vara,\varb,\varc,\vard,\vare) = \\
				&\quad\sum_{k_{1},k_{2},k_{3},k_{4}\geq 0} \prod_{i=1}^{2}\qbinom{\vare-2\vara+2\varb+\varc+\vard-k_{3}-k_{4}}{k_{i}}_q \qbinom{\vare-2\vara+\varb+\varc+\vard}{k_{i+2}}_q  \\[0.1cm]
				& \hspace{3cm} \times \tilde{\Gfct}(\vara-\varb-k_{1} ,\varb-k_{1}-k_{2}-k_{3}-k_{4} ,\varc-k_{1}-k_{4} ,\vard-k_{1}-k_{3} ,\vare)\,.
			\end{split}
			\label{eq:Gtilderec}
		\end{equation}
    \label{lem:Gtilderec}
\end{prop}
  
%		This \namecref{thm:qPS} will turn out to be a crucial ingredient for the induction step in the proof of the \cref{prop:closedform}.

	\begin{proof}
		
	In order to prove \eqref{eq:Gtilderec} we will repeatedly use \cref{thm:qPS}. We start from the right-hand side of \eqref{eq:Gtilderec}, plug in the definition of $\tilde{\Gfct}$ given in \eqref{eq:Gtildedef}, expand the binomials, and collect all factorials involving $k_1$ and $k_2$. Doing so one finds for the r.h.s.\ of \eqref{eq:Gtilderec} that
		\begin{equation*}
			\begin{split}
				&\sum_{k_{3},k_{4}\geq 0} \qbinom{\varb-2\vara + \varc + \vard + \vare}{k_3}_q \qbinom{\varb-2\vara + \varc + \vard + \vare}{k_4}_q\\
				& \quad \times\frac{[\vara - \varb - \vard + k_3]_q! \, [2\varb-2\vara + \varc + \vard + \vare - k_3 - k_4]_q!^2 \, [
					\vara - \varb - \varc + k_4]_q!}{[\varb-\vara + \vard + \vare - k_3]_q! \, [\varb-\vara + \varc + 
					\vare - k_4]_q!} \\[0.8em]
				&\quad \times \sum_{k_{1}\geq 0} \frac{ [\varb-\vara + \varc + \vard + \vare - k_1 - k_3 - 
					k_4]_q!}{ [k_1]_q! \, [\varc - k_1 - k_4]_q! \, [\vard - k_1 - k_3]_q!} \\[-0.2em]
				&\quad \hspace{2cm} \times \frac{1}{[2\varb-2\vara + \varc + \vard + \vare - k_1 - k_3 - k_4]_q! \, [\vara - \varb - \varc - \vard + k_1 + k_3 + k_4]_q!}\\[0.8em]
				&\quad \times \sum_{k_{2}\geq 0} \frac{ [2\varb - \vara + \vare - k_2 - k_3 - k_4]_q!}{ [k_2]_q! \, [\varb - \varc - k_2 - k_3]_q! \, [\varb - \vard - k_2 - k_4]_q!}\\[-0.2em]
				&\quad \hspace{2cm} \times \frac{1}{[2\varb-2\vara + \varc + \vard + \vare - k_2 - k_3 - k_4]_q! \, [\vara - 2 \varb + k_2 + k_3 + k_4]_q!}\,.
			\end{split}
		\end{equation*}
		We can now use \eqref{eq:qPS} to perform the sum over $k_1$ and $k_2$. Collecting the factorials depending on $k_3$ and $k_4$ in the resulting expression, we get
		\begin{equation*}
			\begin{split}
				& \frac{[\vara - \varb]_q! \,[\vara - \varc - \vard]_q! \,[\varb-2\vara + \varc + \vard + \vare]_q!^2}{ [\varb-\vara + \vare]_q! \, [\varc-\vara + \vard + \vare]_q!} \\
				& \quad \times \sum_{k_3 \geq 0} \frac{ [\varb-\vara + \vard + \vare - k_3]_q!}{ [k_3]_q! \, [\vard - k_3]_q! \, [
					\varb - \varc - k_3]_q! \,  [\varb-2\vara + \varc + \vard + \vare - k_3]_q! \,  [
					\vara - \varb - \vard + k_3]_q!} \\
				& \quad \times\sum_{k_4 \geq 0} \frac{ [\varb-\vara + \varc + \vare - k_4]_q!}{ [k_4]_q! \, [\varc - k_4]_q! \, [
					\varb - \vard - k_4]_q! \, [\varb-2\vara + \varc + \vard + \vare - k_4]_q! \,  [
					\vara - \varb - \varc + k_4]_q!}\,.
			\end{split}
		\end{equation*}
		Thus, we can use \eqref{eq:qPS} again to carry out the sums over $k_3$ and $k_4$ to deduce that the right-hand side of \eqref{eq:Gtilderec} equals
		\begin{equation*}
			\frac{ [\vara - \varc]_q! \, [\vara - \vard]_q! \, [\varb-\vara + \vare]_q! \, [\varc-\vara + \vard + \vare]_q!}{ [\vara - \varb]_q! \, [\varb - \varc]_q! \, [\varc]_q! \, [\varb - \vard]_q! \, [\vara - \varc - \vard]_q! \, [\vard]_q! \, [\varc-\vara + \vare]_q! \, [\vard - \vara + \vare]_q!}\,.
		\end{equation*}
		The above is exactly the expansion of $\tilde{\Gfct}(\vara,\varb,\varc,\vard,\vare)$ into factorials, proving  \eqref{eq:Gtilderec}.
		
	\end{proof}

 Using \cref{lem:Gtilderec} we can conclude the proof of \cref{prop:Gclosedform}.
 
\begin{proof}[Proof of \cref{prop:Gclosedform}]
By \eqref{eq:Grec}, for $a>0$ and $b>0$, $G(a,b,c,d,e)$ is determined by the value of $G(a',b',c',d',e)$ for $0\leq a<a'$. This means that $G(a,b,c,d,e)$ is the unique solution of \eqref{eq:Grec} compatible with the boundary value in \eqref{eq:Gab0} for $G(0,b,c,d,e)$. By the second part of \cref{lem:Gab0}, the r.h.s. of \eqref{eq:Gclosedform} indeed returns the correct value of $G(0,b,c,d,e)$ and $G(a,0,c,d,e)$ from \eqref{eq:Gdef}, so all that is left to do is to check that it solves \eqref{eq:Grec}, and conclude by uniqueness.

To this aim, it is sufficient to argue using \cref{lem:Gtilderec}: note first of all that the coefficients of the recursion are invariant under the shift $(\vara,\vare) \rightarrow (\vara-1,\vare-2)$. Therefore, since 
		$\tilde{\Gfct}(\vara,\varb,\varc,\vard,\vare)$ is a solution of \eqref{eq:Grec} by \cref{lem:Gtilderec},  then so is
		$\tilde{\Gfct}(\vara-1,\varb,\varc,\vard,\vare-2)$. Furthermore, since \eqref{eq:Grec} is linear in $\Gfct$, their difference is also a solution. But by \eqref{eq:Gtildedef} this is exactly the claimed expression for $\Gfct(\vara,\varb,\varc,\vard,\vare)$ in \eqref{eq:Gclosedform}, which concludes the proof.\\ 
\end{proof}

	\begin{cor}
	\cref{conj:higherg} holds for any quasi-tame pair $Y(D)$.
	\label{cor:proof}
	\end{cor}

    \begin{proof}
            
    The tame case having been treated in \cite{Bousseau:2020fus}, it suffices to restrict to non-tame, quasi-tame pairs. Taking the specialisation \eqref{eq:Nlogred} of \eqref{eq:Gclosedform}  leads to  \eqref{eq:NlogdP302closedform}, which together with  \eqref{eq:opendp302} establishes the first equality of \cref{conj:higherg} for $Y(D)=\mathrm{dP}_3(0,2)$. Since $\bbO_d\l(\mathrm{dP}^{\rm op}_3(0,2)\r)=\bbO_d\l(\mathrm{dP}^{\rm op}_3(1,1)\r)$ \cite[Sec.~6.3.1]{Bousseau:2020fus}, the BPS integrality statement in the second equality further follows without any modification from the proof of \cite[Thm.~8.1]{Bousseau:2020fus} for $l=2$. Finally, since every non-tame, quasi-tame pair $Y(D)$ is related to $\mathrm{dP}_3(0,2)$ by a series of $m\geq 0$ iterated interior blow-ups at general points of $D$ \cite[Prop.~2.2]{Bousseau:2020fus}, \cref{prop:birinv} further implies that \cref{conj:higherg} holds for any such pair. 
	
	\end{proof}	
	    
	%\newpage

		\bibliography{refs}

\def\cprime{$'$} \def\cprime{$'$}
\begin{bibdiv}
\begin{biblist}

\bib{AbramChen14}{article}{
      author={Abramovich, D.},
      author={Chen, Q.},
       title={Stable logarithmic maps to {D}eligne--{F}altings pairs {II}},
        date={2014},
        ISSN={1093-6106},
     journal={Asian J. Math.},
      volume={18},
      number={3},
       pages={465\ndash 488},
  url={https://0-doi-org.pugwash.lib.warwick.ac.uk/10.4310/AJM.2014.v18.n3.a5},
}

\bib{Bardwell-Evans:2021tsz}{article}{
      author={Bardwell-Evans, S.},
      author={Cheung, M-W.~M.},
      author={Hong, H.},
      author={Lin, Y-S.},
       title={{Scattering Diagrams from Holomorphic Discs in Log Calabi-Yau
  Surfaces}},
        date={2021},
      eprint={2110.15234},
}

\bib{barrott2020explicit}{article}{
      author={Barrott, L.~J.},
       title={Explicit equations for mirror families to log {C}alabi--{Y}au
  surfaces},
        date={2020},
     journal={Bull. Korean Math. Soc.},
      volume={57},
      number={1},
       pages={139\ndash 165},
}

\bib{MR4048291}{article}{
      author={Bousseau, P.},
       title={Quantum mirrors of log {C}alabi--{Y}au surfaces and higher-genus
  curve counting},
        date={2020},
        ISSN={0010-437X},
     journal={Compos. Math.},
      volume={156},
      number={2},
       pages={360\ndash 411},
         url={https://doi.org/10.1112/s0010437x19007760},
}

\bib{bousseau2018quantum}{article}{
      author={Bousseau, P.},
       title={The quantum tropical vertex},
        date={2020},
        ISSN={1465-3060},
     journal={Geom. Topol.},
      volume={24},
      number={3},
       pages={1297\ndash 1379},
         url={https://doi.org/10.2140/gt.2020.24.1297},
}

\bib{bousseau2020strong}{article}{
      author={Bousseau, P.},
       title={Strong positivity for the skein algebras of the $4 $-punctured
  sphere and of the $1 $-punctured torus},
        date={2022},
     journal={Commun.~Math.~Phys.},
       pages={1\ndash 58},
      eprint={2009.02266},
}

\bib{Bousseau:2020fus}{article}{
      author={Bousseau, P.},
      author={Brini, A.},
      author={van Garrel, M.},
       title={Stable maps to {L}ooijenga pairs},
     journal={Geom.~Topol.},
      volume={({\rm in press})},
      eprint={2011.08830},
}

\bib{Bousseau:2020ryp}{article}{
      author={Bousseau, P.},
      author={Brini, A.},
      author={van Garrel, M.},
       title={{Stable maps to {L}ooijenga pairs: orbifold examples}},
        date={2021},
     journal={Lett. Math. Phys.},
      volume={111},
       pages={109},
      eprint={2012.10353},
}

\bib{Bousseau:2019bwv}{article}{
      author={Bousseau, P.},
      author={Brini, A.},
      author={van Garrel, M.},
       title={{On the log-local principle for the toric boundary}},
        date={2022},
     journal={Bull.~Lond.~Math.~Soc.},
      volume={54},
      number={1},
       pages={161\ndash 181},
      eprint={1908.04371},
}

\bib{Chen14}{article}{
      author={Chen, Q.},
       title={Stable logarithmic maps to {D}eligne--{F}altings pairs {I}},
        date={2014},
        ISSN={0003-486X},
     journal={Ann. of Math. (2)},
      volume={180},
      number={2},
       pages={455\ndash 521},
  url={https://0-doi-org.pugwash.lib.warwick.ac.uk/10.4007/annals.2014.180.2.2},
}

\bib{Chiang:1999tz}{article}{
      author={Chiang, T.M.},
      author={Klemm, A.},
      author={Yau, S-T.},
      author={Zaslow, E.},
       title={{Local mirror symmetry: Calculations and interpretations}},
        date={1999},
     journal={Adv. Theor. Math. Phys.},
      volume={3},
       pages={495\ndash 565},
      eprint={hep-th/9903053},
}

\bib{MR2276766}{article}{
      author={Coates, T.},
      author={Givental, A.},
       title={Quantum {R}iemann-{R}och, {L}efschetz and {S}erre},
        date={2007},
        ISSN={0003-486X},
     journal={Ann. of Math. (2)},
      volume={165},
      number={1},
       pages={15\ndash 53},
         url={http://dx.doi.org/10.4007/annals.2007.165.15},
}

\bib{davison2019strong}{article}{
      author={Davison, B.},
      author={Mandel, T.},
       title={Strong positivity for quantum theta bases of quantum cluster
  algebras},
        date={2021},
     journal={Invent. Math.},
      volume={226},
       pages={725\ndash 843},
      eprint={1910.12915},
}

\bib{Fang:2011qd}{article}{
      author={Fang, B.},
      author={Liu, C-C.~M.},
       title={{Open {G}romov--{W}itten invariants of toric {C}alabi--{Y}au
  3-folds}},
        date={2013},
     journal={Commun. Math. Phys.},
      volume={323},
       pages={285\ndash 328},
      eprint={1103.0693},
}

\bib{Fri16}{article}{
      author={Friedman, R.},
       title={On the geometry of anticanonical pairs},
        date={2016},
      eprint={1502.02560},
}

\bib{vGGR}{article}{
      author={Garrel, M.~van},
      author={Graber, T.},
      author={Ruddat, H.},
       title={Local {G}romov--{W}itten invariants are log invariants},
        date={2019},
        ISSN={0001-8708},
     journal={Adv. Math.},
      volume={350},
       pages={860\ndash 876},
  url={https://0-doi-org.pugwash.lib.warwick.ac.uk/10.1016/j.aim.2019.04.063},
}

\bib{grafnitz:2020tropical}{article}{
      author={Gr\"afnitz, T.},
       title={{Tropical correspondence for smooth del Pezzo log Calabi-Yau
  pairs}},
        date={2022},
     journal={J. Alg. Geom.},
      volume={31},
      number={4},
       pages={687\ndash 749},
      eprint={2005.14018},
}

\bib{GRZ:2022properLG}{article}{
      author={Gr\"afnitz, T.},
      author={Ruddat, H.},
      author={Zaslow, E.},
       title={{The proper Landau-Ginzburg potential is the open mirror map}},
        date={2022},
      eprint={2204.12249},
}

\bib{GHKclu}{article}{
      author={Gross, M.},
      author={Hacking, P.},
      author={Keel, S.},
       title={Birational geometry of cluster algebras},
        date={2015},
        ISSN={2214-2584},
     journal={Algebr. Geom.},
      volume={2},
      number={2},
       pages={137\ndash 175},
  url={https://0-doi-org.pugwash.lib.warwick.ac.uk/10.14231/AG-2015-007},
}

\bib{GHKlog}{article}{
      author={Gross, M.},
      author={Hacking, P.},
      author={Keel, S.},
       title={Mirror symmetry for log {C}alabi--{Y}au surfaces {I}},
        date={2015},
        ISSN={0073-8301},
     journal={Publ. Math. Inst. Hautes \'{E}tudes Sci.},
      volume={122},
       pages={65\ndash 168},
  url={https://0-doi-org.pugwash.lib.warwick.ac.uk/10.1007/s10240-015-0073-1},
}

\bib{GPS}{article}{
      author={Gross, M.},
      author={Pandharipande, R.},
      author={Siebert, B.},
       title={The tropical vertex},
        date={2010},
        ISSN={0012-7094},
     journal={Duke Math. J.},
      volume={153},
      number={2},
       pages={297\ndash 362},
  url={https://0-doi-org.pugwash.lib.warwick.ac.uk/10.1215/00127094-2010-025},
}

\bib{GS13}{article}{
      author={Gross, M.},
      author={Siebert, B.},
       title={Logarithmic {G}romov--{W}itten invariants},
        date={2013},
        ISSN={0894-0347},
     journal={J. Amer. Math. Soc.},
      volume={26},
      number={2},
       pages={451\ndash 510},
  url={https://0-doi-org.pugwash.lib.warwick.ac.uk/10.1090/S0894-0347-2012-00757-7},
}

\bib{hacking2020homological}{article}{
      author={Hacking, P.},
      author={Keating, A.},
       title={Homological mirror symmetry for log {C}alabi--{Y}au surfaces},
        date={2020},
      eprint={2005.05010},
}

\bib{MR3739228}{article}{
      author={Ionel, E.~N.},
      author={Parker, T.~H.},
       title={The {G}opakumar--{V}afa formula for symplectic manifolds},
        date={2018},
        ISSN={0003-486X},
     journal={Ann. of Math. (2)},
      volume={187},
      number={1},
       pages={1\ndash 64},
         url={https://doi.org/10.4007/annals.2018.187.1.1},
}

\bib{Iqbal:2004ne}{article}{
      author={Iqbal, A.},
      author={Kashani-Poor, A-K.},
       title={{The vertex on a strip}},
        date={2006},
     journal={Adv. Theor. Math. Phys.},
      volume={10},
      number={3},
       pages={317\ndash 343},
      eprint={hep-th/0410174},
}

\bib{Katz:2001vm}{article}{
      author={Katz, S.~H.},
      author={Liu, C-C.~M.},
       title={{Enumerative geometry of stable maps with Lagrangian boundary
  conditions and multiple covers of the disc}},
        date={2002},
     journal={Geom. Topol. Monogr.},
      volume={8},
       pages={1\ndash 47},
      eprint={math/0103074},
}

\bib{Klemm:2007in}{article}{
      author={Klemm, A.},
      author={Pandharipande, R.},
       title={{Enumerative geometry of Calabi--Yau 4-folds}},
        date={2008},
     journal={Commun. Math. Phys.},
      volume={281},
       pages={621\ndash 653},
      eprint={math/0702189},
}

\bib{MR2380471}{article}{
      author={Konishi, Y.},
      author={Minabe, S.},
       title={Flop invariance of the topological vertex},
        date={2008},
        ISSN={0129-167X},
     journal={Internat. J. Math.},
      volume={19},
      number={1},
       pages={27\ndash 45},
         url={https://doi.org/10.1142/S0129167X08004546},
}

\bib{MR2181810}{incollection}{
      author={Kontsevich, M.},
      author={Soibelman, Y.},
       title={Affine structures and non-{A}rchimedean analytic spaces},
        date={2006},
   booktitle={The unity of mathematics},
      series={Progr. Math.},
      volume={244},
   publisher={Birkh\"{a}user Boston, Boston, MA},
       pages={321\ndash 385},
         url={https://doi.org/10.1007/0-8176-4467-9_9},
}

\bib{Kra21}{article}{
      author={Krattenthaler, C.},
       title={Proof of two multivariate $q$-binomial sums arising in
  {G}romov--{W}itten theory},
        date={2021},
      eprint={2102.02360},
}

\bib{Labastida:2000zp}{article}{
      author={Labastida, J. M.~F.},
      author={Marino, M.},
       title={{Polynomial invariants for torus knots and topological strings}},
        date={2001},
     journal={Commun. Math. Phys.},
      volume={217},
       pages={423\ndash 449},
      eprint={hep-th/0004196},
}

\bib{Labastida:2000yw}{article}{
      author={Labastida, J. M.~F.},
      author={Marino, M.},
      author={Vafa, C.},
       title={{Knots, links and branes at large~$N$}},
        date={2000},
     journal={JHEP},
      volume={11},
       pages={007},
      eprint={hep-th/0010102},
}

\bib{Li:2004uf}{article}{
      author={Li, J.},
      author={Liu, C-C.~M.},
      author={Liu, K.},
      author={Zhou, J.},
       title={{A mathematical theory of the topological vertex}},
        date={2009},
     journal={Geom. Topol.},
      volume={13},
       pages={527\ndash 621},
      eprint={math/0408426},
}

\bib{Li:2001sg}{article}{
      author={Li, J.},
      author={Song, Y.~S.},
       title={{Open string instantons and relative stable morphisms}},
        date={2002},
     journal={Adv.Theor.Math.Phys.},
      volume={5},
       pages={67\ndash 91},
      eprint={hep-th/0103100},
}

\bib{Liu:2021eeb}{article}{
      author={Liu, C-C~M.},
      author={Yu, S.},
       title={{Open/closed correspondence via relative/local correspondence}},
        date={2022},
     journal={Adv. Math.},
      volume={410},
       pages={108696},
      eprint={2112.04418},
}

\bib{MR3518552}{article}{
      author={Mandel, T.},
       title={Tropical theta functions and log {C}alabi--{Y}au surfaces},
        date={2016},
        ISSN={1022-1824},
     journal={Selecta Math. (N.S.)},
      volume={22},
      number={3},
       pages={1289\ndash 1335},
         url={https://doi.org/10.1007/s00029-015-0221-y},
}

\bib{MR3954363}{article}{
      author={Mandel, T.},
       title={Classification of rank 2 cluster varieties},
        date={2019},
     journal={SIGMA},
      volume={15},
       pages={Paper 042, 32},
         url={https://doi.org/10.3842/SIGMA.2019.042},
}

\bib{mandel2015scattering}{article}{
      author={Mandel, T.},
       title={Scattering diagrams, theta functions, and refined tropical curve
  counts},
        date={2021},
     journal={J.~London~Math.~Soc.},
      volume={104},
      number={5},
       pages={2299\ndash 2334},
      eprint={1503.06183},
}

\bib{Man19}{article}{
      author={Mandel, T.},
       title={Theta bases and log {G}romov--{W}itten invariants of cluster
  varieties},
        date={2021},
     journal={Trans.~Amer.~Math.~Soc.},
      volume={374},
      number={08},
       pages={5433\ndash 5471},
      eprint={1903.03042},
}

\bib{ManRu}{article}{
      author={Mandel, T.},
      author={Ruddat, H.},
       title={Descendant log {G}romov--{W}itten invariants for toric varieties
  and tropical curves},
        date={2020},
        ISSN={0002-9947},
     journal={Trans. Amer. Math. Soc.},
      volume={373},
      number={2},
       pages={1109\ndash 1152},
         url={https://doi.org/10.1090/tran/7936},
}

\bib{Ooguri:1999bv}{article}{
      author={Ooguri, H.},
      author={Vafa, C.},
       title={{Knot invariants and topological strings}},
        date={2000},
     journal={Nucl. Phys.},
      volume={B577},
       pages={419\ndash 438},
      eprint={hep-th/9912123},
}

\bib{MR3563248}{article}{
      author={Yu, T.~Y.},
       title={Enumeration of holomorphic cylinders in log {C}alabi--{Y}au
  surfaces. {I}},
        date={2016},
        ISSN={0025-5831},
     journal={Math. Ann.},
      volume={366},
      number={3-4},
       pages={1649\ndash 1675},
         url={https://doi.org/10.1007/s00208-016-1376-3},
}

\bib{yu2016enumeration}{article}{
      author={Yu, T.~Y.},
       title={Enumeration of holomorphic cylinders in log {C}alabi--{Y}au
  surfaces. {II}. {P}ositivity, integrality and the gluing formula},
        date={2021},
     journal={Geom. Topol.},
      volume={25},
      number={1},
       pages={1\ndash 46},
      eprint={1608.07651},
}

\bib{ZeilSaal}{article}{
      author={Zeilberger, D.},
       title={A {$q$}-{F}oata proof of the {$q$}-{S}aalsch\"{u}tz identity},
        date={1987},
        ISSN={0195-6698},
     journal={European J. Combin.},
      volume={8},
      number={4},
       pages={461\ndash 463},
  url={https://0-doi-org.pugwash.lib.warwick.ac.uk/10.1016/S0195-6698(87)80054-9},
}

\bib{zhou2019weyl}{article}{
      author={Zhou, Y.},
       title={Weyl groups and cluster structures of families of log
  {C}alabi--{Y}au surfaces},
        ISSN={1073-7928},
     journal={Int.~Math.~Res.~Not.},
      volume={({\rm in press})},
      eprint={1910.05762},
         url={https://doi.org/10.1093/imrn/rnac229},
}

\end{biblist}
\end{bibdiv}
	
%	\printbibliography[heading=bibintoc]

\end{document}